\documentclass{article}
\usepackage[letterpaper]{geometry}
\usepackage[T1]{fontenc}
\usepackage{amsmath,amssymb,amsthm}
\usepackage{stmaryrd}
\SetSymbolFont{stmry}{bold}{U}{stmry}{m}{n}
\usepackage{graphicx}
\usepackage{algorithm}
\usepackage{physics}
\usepackage[numbers,square,sort&compress]{natbib}
\usepackage[hidelinks]{hyperref}
\usepackage{cleveref}
\usepackage{microtype}
\usepackage{algpseudocode}
\usepackage{booktabs} 
\usepackage{wrapfig}
\usepackage{tcolorbox}
\usepackage{tikz}
\usepackage{enumerate}
\usetikzlibrary{arrows.meta,graphs,shapes}

\newcommand{\E}{{\mathbb{E}}}

\newcommand{\NN}{\mathbb{N}}

\newcommand{\threval}{\mathsf{Three\text{-}valued}}

\newcommand{\ADASGD}[0]{\textsc{AdaSGD}}

\newcommand{\dist}[2]{\mathbf{dist}(#1,#2)}
\newcommand{\R}{\mathbb{R}}
\newcommand{\littleFunc}[0]{\rho}

\DeclareMathOperator*{\argmin}{arg\,min}

\newcommand{\simIID}{\stackrel{\mathrm{iid}}{\sim}}

\newcommand{\ball}[1]{\mathbb{B}(#1)}
\newcommand{\ind}[1]{^{(#1)}}
\newcommand{\independent}{\perp\!\!\!\perp}
\newcommand{\budget}{n}
\newcommand{\pr}{\mathbb{P}}
\newcommand{\BudgetLowerBoundPhi}{W}
\newcommand{\Fbudget}{m}
\newcommand{\kn}{K_\budget}

\newcommand{\xsgd}[1]{x^{\text{SGD}}_{#1}}
\newcommand{\xout}[1]{x^{\text{out}}_{#1}}
\newcommand{\xStageOne}[1]{x_{#1}^{\text{stage 1}}}
\newcommand{\xB}[0]{\xStageOne{\Fbudget}}
\newcommand{\gsgd}[0]{g_{t}}

\newcommand{\nonasymFunc}{\beta}
\newcommand{\dimFunc}{d}
\newcommand{\xt}{x^{(t)}}
\newcommand{\alg}{\text{A}}
\newcommand{\algcls}{\mathcal{A}}
\newcommand{\funcls}{\mathcal{F}}
\newcommand{\funclsncvx}[1][d]{\funcls^{p,\text{ncvx}}_{#1}}
\newcommand{\funclscvx}[1][d]{\funcls^{p,\text{cvx}}_{#1}}
\newcommand{\funclscvxhil}{\funcls^{p,\text{cvx}}_\infty}
\newcommand{\funclsncvxhil}{\funcls^{p,\text{ncvx}}_\infty}
\newcommand{\algclszr}{\algcls_{\text{zr}}}
\newcommand{\algclsrand}{\algcls_{\text{rand}}}
\newcommand{\xcom}[1]{\underline{x}_{#1}}
\newcommand{\ycom}[1]{\underline{y}_{#1}}
\newcommand{\funcom}[1]{f_{#1}}
\newcommand{\lipsfunc}[2]{\mathcal{C}^{#2}(#1)}
\newcommand{\probmeasure}{\mathcal{D}}
\newcommand{\hilbert}[1]{\mathcal{H}_{#1}}
\newcommand{\algorder}{\bar{p}}

\newcommand{\indicator}[1]{\mathbb{I}\{#1\}}

\newtheorem{theorem}{Theorem}
\newtheorem{lemma}{Lemma}

\newtheorem{corollary}{Corollary}

\newtheorem{assumption}{Assumption}
\newtheorem{definition}{Definition}
\newtheorem{example}{Example}
\newtheorem{remark}{Remark}
\newtheorem{fact}{Fact}
\crefname{fact}{fact}{facts}
\Crefname{fact}{Fact}{Facts}

\hypersetup{
  pdftitle={Asymptotic Lower Bounds for Continuous Optimization},
  pdfauthor={Oliver Hinder and Yuntian Jiang}
}
\title{Asymptotic Lower Bounds for Continuous Optimization}
\author{Oliver Hinder\thanks{Department of Industrial Engineering, University of Pittsburgh.
  Email: \texttt{ohinder@pitt.edu}.}
  \and Yuntian Jiang\thanks{Research Institute for Interdisciplinary Sciences, Shanghai University of Finance and Economics. This work was performed when Yuntian Jiang was a visiting student at Department of Industrial Engineering of University of Pittsburgh. Email: \texttt{yuntianjiang07@gmail.com}.}}
\date{\today}

\begin{document}
\maketitle
\begin{abstract}
Nonasymptotic convergence rates for optimization problems have been
extensively studied across a wide range of settings, with carefully
constructed worst-case instances establishing matching lower bounds for many
algorithm classes. Recent work has shown, however, that these rates can often
be improved in the asymptotic regime, while the corresponding asymptotic lower
bounds remain largely unknown.
We propose a general technique for converting existing nonasymptotic lower bound
constructions into asymptotic lower bounds on Hilbert spaces. This allows us to
show that several known asymptotic convergence upper bounds in Hilbert spaces
are tight. In particular, we recover tightness of the $o(n^{-2})$
suboptimality guarantee for accelerated gradient descent on smooth convex
functions \cite{attouch2016rate}, as well as the $o(n^{-1/2})$
suboptimality guarantee for gradient descent on smooth nonconvex functions
\cite{gratton2025refining}.
For stochastic convex optimization, we also develop a method that achieves an
$o(\budget^{-1/2})$ suboptimality guarantee in finite dimensions, and prove a
matching one-dimensional lower bound.
\end{abstract}

\section{Introduction}

In optimization, there are two popular ways to measure the performance of a method: asymptotic and nonasymptotic convergence bounds. 

Asymptotic convergence rates only apply when the number of iterations, $\budget$, grows arbitrarily large.
Given a function $s : \NN \rightarrow (0,\infty)$,
a method has an $O(s(\budget))$ asymptotic convergence rate, if for every problem instance, 
\[
\limsup_{\budget \rightarrow \infty} \frac{\epsilon_\budget}{s(\budget)} < \infty
\]
where $\epsilon_\budget \ge 0$ is the value of our desired performance measure after $\budget$ iterations (i.e., $\epsilon_{\budget}$ 
is typically the suboptimality in convex optimization and the gradient norm in nonconvex optimization). 
If the limit is zero then we say that the method has an $o(s(\budget))$ asymptotic convergence rate. 
Asymptotic convergence rates are popular in classical nonlinear optimization \citep{wright1999numerical,dennis1974characterization,Bertsekas2016Nonlinear} 
and statistics \citep{van2000asymptotic}.

On the other hand, nonasymptotic convergence bounds are stronger and apply for all possible inputs, not just in the limit as $\budget \rightarrow \infty$.
In particular, given a function $s : \NN \times \Theta \rightarrow (0,\infty)$, a method has an $O(s(\budget, \theta))$ nonasymptotic convergence bound if there exists a problem-independent constant $C \in (0,\infty)$ such that, for every iteration budget, $\budget$, and value of the problem parameters $\theta$,
\[
\frac{\epsilon_\budget}{s(\budget, \theta)} \le C.
\]
The nonasymptotic complexity of optimization is well-understood \citep{nemirovskii1983problem,carmon2020lower,arjevani2019oracle}.
For example, the minimax nonasymptotic complexity of stochastic convex optimization with Lipschitz continuous functions is known to be $O( L R / \sqrt{\budget} )$ for the expected suboptimality where $\budget$ is the number of stochastic gradient evaluations, $R$ is the distance to optimality and $L$ is the Lipschitz constant of sample functions.

The asymptotic and nonasymptotic complexities of optimization need not coincide.
For example, for smooth convex functions, accelerated gradient descent achieves
an $o(n^{-2})$ suboptimality guarantee \citep{attouch2016rate}, improving over
the classical nonasymptotic guarantee of $O(n^{-2})$. In some settings, the
difference can be even more pronounced. For instance, on stochastic piecewise
linear convex functions, asymptotic linear convergence rates are achievable,
whereas the best available nonasymptotic convergence rates are sublinear \citep{davis2024stochastic}.
This paper demonstrates, for several canonical settings for analyzing optimization methods,
that there is little to gain by switching from nonasymptotic to asymptotic analysis.

\paragraph{Our contributions} 
\begin{itemize}
\item We develop a general technique for converting nonasymptotic lower-bound
constructions into asymptotic lower bounds. The key idea is to combine a
sequence of finite-dimensional hard instances into a single hard problem, so that
the difficulty persists along an infinite sequence of iteration
budgets. We use this technique to prove tightness of known asymptotic upper bounds
for deterministic optimization. In particular, we show that the
$o(n^{-2})$ suboptimality guarantee for accelerated gradient descent on smooth
convex functions \citep{attouch2016rate} cannot be improved in general, and
that the $o(n^{-1/2})$ suboptimality guarantee for gradient descent on smooth
nonconvex functions \citep{gratton2025refining} is also tight.

\item For stochastic convex optimization, we develop a finite-dimensional method
that achieves an $o(\budget^{-1/2})$ expected suboptimality guarantee. We also
prove a matching one-dimensional lower bound by nesting an infinite sequence of
existing nonasymptotic lower bounds into one dimension, showing that this asymptotic rate
cannot be improved in general.
\end{itemize}

\paragraph{Paper outline}
The remainder of the introduction surveys the related literature, and introduces our notation and assumptions. 
\Cref{sec:general-lower-bounds} provides a general technique to convert nonasymptotic lower bounds to asymptotic lower bounds in Hilbert spaces.
\Cref{sec:upper-bounds} provides our parameter-free method with an improved asymptotic convergence rate for stochastic convex optimization and \Cref{sec.lower bound} provides a matching lower bound for stochastic convex optimization.
\Cref{sec:discussion} discusses the disadvantages of asymptotic convergence rates compared to nonasymptotic rates.

\subsection{Related literature}
\paragraph{Asymptotic convergence bounds for deterministic optimization}
Several papers study asymptotic convergence bounds in the deterministic setting.
In the H\"older continuous optimization setting, \citet{nesterov2025universal} showed that, with knowledge of the local geometric properties of the objective function, improved asymptotic convergence bounds are achievable. In particular, when the H\"older constant equals $0$, an asymptotic convergence bound of $o(\budget^{-1/2})$ can be obtained.
In follow-up work, \citet[Theorem~4.4]{he2025new} showed that similar results hold for Polyak’s stepsize~\citep{polyak1987introduction}. 
%We note that these methods rely on knowledge of the problem parameters or minimum function value, and are therefore not parameter-free.
For smooth minimization problems with a simple nonsmooth term, \citet{attouch2016rate} showed that the asymptotic convergence rate of Nesterov’s accelerated method~\citep{nesterov1983method} is in fact $o(\budget^{-2})$, improving upon the original result. Later this result was extended to linearly constrained convex optimization~\citep{he2026convergence} and convex optimization with capped-$\ell$1 penalty~\citep{tao2026accelerated}.
For nonconvex optimization, \citet{gratton2025refining} showed that gradient descent attains a complexity of $o(\budget^{-1/2})$ for finding stationary points, and their results were further extended to difference-of-convex optimization \citep{gratton2026iteration}.

\paragraph{Nonasymptotic worst-case lower bounds}
Nonasymptotic lower bounds for convex optimization can be traced back to the seminal work of \citet{nemirovskii1983problem}. These results were subsequently extended and refined for stochastic optimization \citep{agarwal2009information}. While there are papers showing nonasymptotic complexity bounds for various settings \citep{ramdas2012optimal,carmon2020lower,carmon2021lower,kornowski2022oracle,zhang2022lower,arjevani2023lower,carmon2024price,GUZMAN20151}, we are not aware of any asymptotic lower bounds in the optimization literature.

\paragraph{Nonasymptotic worst-case upper bounds for stochastic optimization}
Nonasymptotic upper bounds for stochastic optimization in the convex setting have been extensively studied and can be traced back to the seminal work of \citet{nemirovskii1983problem}. In the strongly convex setting, improved convergence bounds are available; see, for example, \citet{hazan2014beyond, rakhlin2011making}. The analyses for convex and strongly convex objectives can be further unified under more general algorithmic frameworks, as demonstrated by \citet{ghadimi2012optimal, ghadimi2013optimal}.

\paragraph{Asymptotic convergence bounds for stochastic optimization}
In the stochastic regime, a line of work investigated settings in which the convergence can be strengthened to almost sure convergence, e.g.,~\citep{sebbouh2021almost,liu2022almost,mertikopoulos2020almost,li2019convergence,zhou2017stochastic}. However, these methods cannot achieve faster convergence than $O(\budget^{-1/2})$. In another line of work more closely related to ours, \citet{polyak1992acceleration} studied the asymptotic convergence of stochastic optimization under smooth and strongly convex assumptions. \citet{duchi2016local} introduced a new notion of ``local minimax complexity'' to study the local geometry property of specific functions, thus allowing a faster asymptotic convergence bound under polynomial growth conditions~\citep[Equation (12)]{duchi2016local}. \citet{ramdas2012optimal} also obtained a faster rate under growth rate conditions and provided lower bounds matching their upper bounds.

\paragraph{Asymptotic characterization of stochastic optimization methods}
While, to the best of our knowledge, there are no papers focusing on the asymptotic rate of convergence of methods for nonsmooth convex stochastic optimization, there are several papers that characterize the asymptotic behavior of stochastic optimization methods under stronger assumptions \citep{davis2024asymptotic,duchi2016localcons,duchi2021asymptotic}.

\subsection{Notation and basic definitions}\label{sec:notation-and-assumptions}

For any $d\in\NN$, let $\hilbert{d}$ denote a $d$-dimensional real
Hilbert space, and let $\hilbert{\infty}$ denote an infinite-dimensional real Hilbert space. We write $\hilbert{}$ for a
generic real Hilbert space whenever the dimension or a particular block
decomposition is immaterial. Throughout the paper, $\langle\cdot,\cdot\rangle$ and $\|\cdot\|$
denote the inner product and induced norm of the ambient Hilbert space,
unless otherwise specified. Let $\R$ denote the reals. When a block decomposition is needed, we fix a sequence
$\{d_i\}_{i\ge 1}\subseteq\NN$ and take
\(
    \hilbert{}
    =
    \bigoplus_{i\ge 1}\hilbert{d_i}
    :=
    \{
        x=(\xcom{i})_{i\ge 1}:
        \xcom{i}\in\hilbert{d_i},\
        \sum_{i\ge 1}\|\xcom{i}\|^2<\infty
    \},
\)
equipped with the inner product defined as
\(
    \langle x,y\rangle
    :=
    \sum_{i\ge 1}\langle \xcom{i},\ycom{i}\rangle,
    x=(\xcom{i})_{i\ge 1},\ y=(\ycom{i})_{i\ge 1}.
\)
For any $p\in\NN$, let
$[p]:=\{1,\ldots,p\}$ and $[\infty]:=\NN$. For $\bar{d}\in \NN \cup \{\infty\}$, and $x\in \hilbert{\bar{d}}$, define
\(
    \operatorname{supp}\{x\}:=\{j\in [\bar{d}]:x_j\neq0\}.
\)
For a symmetric order-$r$ tensor $T$, $r\geq2$, define
\(
    \operatorname{supp}\{T\}
    :=\{j\in[\bar{d}]:T_{j,j_2,\ldots,j_r}\neq0
    \text{ for some }j_2,\ldots,j_r\in [\bar{d}]\}.
\)
For a set $S\subseteq\hilbert{}$ and a point $x\in\hilbert{}$, we define
\(
    \dist{x}{S}:=\inf_{y\in S}\|y-x\|
\)
if $S\neq\emptyset$, and set $\dist{x}{S}:=\infty$ if $S=\emptyset$. We
also use the shorthand $\dist{S}{x}:=\dist{x}{S}$. For a nonempty closed
convex set $X\subseteq\hilbert{}$, the projection of $x\in\hilbert{}$
onto $X$ is denoted by
\(
    \Pi_X(x):=\arg\min_{y\in X}\|y-x\|.
\)
For $x\in\hilbert{}$ and $R>0$, we denote the closed ball centered at
$x$ with radius $R$ by
\(
    \ball{x,R}:=\{u\in\hilbert{}:\|x-u\|\le R\},
\)
and use the shorthand $\ball{R}:=\ball{\mathbf{0},R}$. Throughout,
$\log$ denotes the natural logarithm. Given a function $f : D \rightarrow \R$,
we define $f^\star := \inf_{x \in D} f(x)$. 
When $f$ is convex and nondifferentiable, with 
abuse of notation, we let $\grad f(x)$ denote an arbitrary function with 
$\grad f(x) \in \partial f(x)$.
% \hinder{I added this sentence in to clean things up}

\begin{definition}
    \label{def:lipschitz-function-class}
    Let $d\in\NN\cup\{\infty\}$ and $p\in\{0\}\cup\NN$. With the convention
    $\grad^0 f=f$, we say that a function
    $f:\hilbert{d}\to\mathbb{R}$ has a $L_p$-Lipschitz-continuous
    $p$th-order derivative if $f$ is $p$ times continuously
    differentiable and
    \[
\|\grad^p f(x)-\grad^p f(y)\|\leq L_p\|x-y\|.
    \]
    For $p\geq1$, the norm of $\grad^p f$ is the operator norm on
    symmetric $p$-linear forms; for $p=0$, it is the absolute value.
    For $p=0$, the above condition simply means that $f$ is
    Lipschitz continuous. Throughout this paper, $\lipsfunc{\hilbert{d}}{p}$ denotes the class of such functions with $L_p=1$; in particular, $\lipsfunc{\hilbert{d}}{0}$ denotes the class of $1$-Lipschitz functions.
\end{definition}

\section{A general technique for asymptotic lower bounds in Hilbert spaces}\label{sec:general-lower-bounds}
In this section, we propose a general technique to convert nonasymptotic lower bound results into asymptotic lower bound results on Hilbert spaces~\cite{bauschke2017convex}. We also show that our technique can be applied to many well-known nonasymptotic lower bound results~\citep{carmon2020lower,arjevani2019oracle,woodworth2018graph}.

An algorithm $\alg$ maps an objective function $f: \hilbert{} \to \mathbb{R}$ to a sequence of iterates $\alg[f]$ in $\hilbert{}$, where $\hilbert{}$ is the domain of optimization variables. 
We study two classes of algorithms: zero-respecting algorithms (\Cref{def:zr-sequence-alg}) and 
randomized algorithms (\Cref{def:randomized-alg}).
These definitions closely follow \citet[Section~2.2]{carmon2020lower}.
Zero-respecting algorithms are a stylized class of algorithms that
only move coordinates that at some point exhibited a nonzero gradient. 
They cover the vast majority of existing algorithms and are easier 
to analyze than randomized algorithms.
This makes them useful as a warm-up to randomized algorithms.
A worst-case lower bound for a zero-respecting algorithm applies by a black box reduction to
deterministic methods; see \citet[Section 3.3]{carmon2020lower}.
The only distinguishing feature between
randomized and deterministic algorithms is that randomized algorithms may additionally use a seed $\xi\sim\mathsf{Unif}[0,1]$ drawn independently of $f$.
The notation $\grad^{(0,\ldots,\algorder)}f(x)$ denotes the response of an $\algorder$th-order oracle, consisting of the function value and all derivatives through order $\algorder$ at $x$. 

\begin{definition}[Zero-respecting sequences and algorithms]
    \label{def:zr-sequence-alg}
    Let $\algorder\in \NN$ and $f:\hilbert{}\to \mathbb{R}$ be differentiable through order $\algorder$. We say the sequence $x^{(1)},x^{(2)},\ldots$ is $\algorder$th order zero-respecting with respect to $f$ if
    \begin{equation}
        \label{eq:pth-zr}
        \operatorname{supp}\left\{\xt\right\} \subseteq \bigcup_{r \in[\algorder]} \bigcup_{s<t} \operatorname{supp}\left\{\grad^r f(x^{(s)})\right\} \quad \text { for each } t \in \NN .
    \end{equation}
    An algorithm $\alg$ is $\algorder$th order zero-respecting, i.e., $\alg\in \algclszr^{(\algorder)}$, if for every $f:\hilbert{}\to \mathbb{R}$ differentiable through order $\algorder$, the iterate sequence $\alg[f]$ is $\algorder$th order zero-respecting w.r.t. $f$.
\end{definition}

\begin{definition}[Randomized algorithms]
    \label{def:randomized-alg}
    Let $\algorder\in \NN$. We call $\alg$ a $\algorder$th-order randomized algorithm, i.e., $\alg\in \algclsrand^{(\algorder)}$, if it produces iterates of the following form:
    \begin{equation}
        \label{eq:randomized-alg-update}
        \xt = \alg^{(t)}\left(\xi, \grad^{(0,\ldots,\algorder)} f(x^{(1)}),\ldots, \grad^{(0,\ldots,\algorder)} f(x^{(t-1)})\right)  \text{ for } t\in \NN,
    \end{equation}
    where $\alg^{(t)}$ are measurable mappings into $\hilbert{}$ and $\xi \sim \mathsf{Unif}[0,1]$. 
\end{definition}

We consider both suboptimality gaps (i.e., for convex optimization) and gradient norms (i.e., for nonconvex optimization) to measure algorithmic performance.
To generically cover these two metrics (and other metrics researchers might be interested in), we provide the following unified definition.

\begin{definition}\label{def:error-measure}
Denote the error measure for function $f$ at point $x$ as $\epsilon(x,f)$ where $\epsilon(x,f) \ge 0$ for all $x$ in the domain of $f$. For algorithm $\alg$ and function $f$, denote the error metric with respect to budget $\budget$ as
\begin{equation}\label{eq.min-error-of-alg}
    \epsilon_\budget(\alg,f) = \epsilon(x^{(\budget+1)},f),
\end{equation}
where, for some $\algorder\in\NN$, $\{\xt\}_{t=1}^\infty = \alg[f]$ if $\alg\in\algclszr^{(\algorder)}$ and $\{\xt\}_{t=1}^\infty = \alg[\xi,f]$ if $\alg\in\algclsrand^{(\algorder)}$. Thus, $\alg[f]_\budget:=x^{(\budget+1)}$ denotes the output after $\budget$ oracle calls, with the randomness implicit for a randomized algorithm.
\end{definition}

% \hinder{there ought to be a definition of Lipschitz continuous}

We next introduce the two function classes used in this section. For function classes, a dimension subscript specifies the domain dimension; classes with the same symbol and superscripts have the same defining properties. For each $d\in\NN\cup\{\infty\}$, $\funclsncvx$ and $\funclscvx$ denote the nonconvex and convex classes. 
Recall from \Cref{def:lipschitz-function-class} that $\lipsfunc{\hilbert{d}}{p}$ is the class of functions with $p$th order $1$-Lipschitz derivatives. Our nonconvex class is a normalized version of that in \citet[Definition~1]{carmon2020lower} and the convex class additionally requires convexity and bounded distance from the minimizer to the origin. Scaling can normalize any function to satisfy these properties; thus, for simplicity we only consider these normalized function classes.

\begin{definition}[Normalized function classes]
    \label{def:fun-cls}
For $p\in\NN$, let $\funclsncvx$ and $\funclscvx$ denote the following two classes of functions, with $\funclscvx$ restricted to convex functions:
\[
\begin{aligned}
\funclsncvx &= \{f\in\lipsfunc{\hilbert{d}}{p}: -1\leq f^\star\leq f(\mathbf{0})\leq1,\ \|\grad^r f(\mathbf{0})\|\leq1\ \forall r\in[p]\},\\
\funclscvx &= \{f\in\funclsncvx: f\text{ is convex},\ \dist{\mathbf{0}}{\argmin f}\leq1\}.
\end{aligned}
\]

    % the definition is changed because the hard function in the convex case uses $|x|^{p+1}$, hence not infinitely differentiable. \Cref{example-cvx-class} also changed correspondingly
\end{definition}

Our idea is to place finite-dimensional hard instances in separate blocks of coordinates and sum them to form one function in $\hilbert{\infty}$. The following assumption gives the two properties needed for this construction: the sum remains in the same function class, and the error of the full function is at least a scaled error from each block. The next two examples verify these properties for the nonconvex and convex classes defined above.

\begin{assumption}
\label{assm:error-metric}
Let $\funcls_{d}$ for any $d \in \NN \cup \{ \infty \}$ denote 
a $d$-dimensional function class and let $\funcls_{} := \cup_{d=1}^\infty \funcls_{d}$ with error metric $\epsilon$.
Suppose there exists a positive integer $q$, positive real sequences $\{\lambda_i\}_{i=1}^\infty$ and $\{\mu_i\}_{i=1}^\infty$
such that for any sequence of positive integers
$\{ \dimFunc_i \}_{i=1}^{\infty}$ and of functions 
$\{ f_i \in \mathcal{F}_{\dimFunc_i} \}_{i=1}^{\infty}$ we have $F \in \mathcal{F}_{\infty}$ for
\[
F(x):=\sum_{i=1}^{\infty} \lambda_i \funcom{i}(\mu_i \xcom{i}), \quad \forall x \in \hilbert{\infty} \text{ where } x := (\xcom{i})_{i \ge 1},
\]
and, for all $i \in \NN$
\begin{equation}
\label{eq:error-metric}
\epsilon(x,F)
\ge
\lambda_i\mu_i^q\epsilon(\mu_i \xcom{i},\funcom{i}).
\end{equation}
\end{assumption}

\begin{example}\label{example-ncvx-class}
    For function class $\funclsncvx$ with error metric $\epsilon(x,f) =\|\grad f(x)\|$, choose $q=1$, $\lambda_i = 2^{-i}$, and $\mu_i=1$. By \Cref{lem:direct-sum-regularity}, $F$ is well defined and satisfies the differentiability, Lipschitz continuity, and derivative bounds at the origin as required by \Cref{def:fun-cls}. By the choice of $\lambda_i$, it is easy to verify that $-1 \le F^\star\le F(\mathbf{0}) \le 1$. Moreover, the block-wise structure gives
    \[
    \|\grad F(x)\|
    =\left(\sum_{j=1}^{\infty}2^{-2j}\|\grad \funcom{j}(\xcom{j})\|^2\right)^{1/2}
    \geq 2^{-i}\|\grad \funcom{i}(\xcom{i})\|.
    \]
    Thus \eqref{eq:error-metric} holds, and therefore \Cref{assm:error-metric} holds for $\funclsncvx$ with $\epsilon(x,f) = \|\grad f(x)\|$.
\end{example}

\begin{example}\label{example-cvx-class}
    For function class $\funclscvx$ and error metric $\epsilon(x,f) =f(x)-f^\star$, let $\xcom{i}^\star\in \argmin \funcom{i}$ with $\|\xcom{i}^\star\|\le 1$, choose $q=0$, $\lambda_i=2^{-\frac{p+3}{2}i}$, and $\mu_i=2^{i/2}$. By \Cref{lem:direct-sum-regularity}, $F$ is well defined and satisfies the differentiability, Lipschitz continuity, and derivative bounds at the origin in \Cref{def:fun-cls}. By construction, $F$ is convex and
\begin{center}
\(
    x^\star = [2^{-1/2}\xcom{1}^\star;2^{-1}\xcom{2}^\star;\ldots]
    \in \argmin F,\quad
    \|x^\star\| \leq \sqrt{\sum_{i=1}^\infty 2^{-i}} = 1, \quad -1\leq F(x^\star)\leq F(\mathbf{0})\leq 1.
\)
\end{center}
Using the block-wise structure, $F(x)-F^\star = \sum_{i=1}^{\infty} \lambda_i \funcom{i}(\mu_i \xcom{i}) - F^\star =  \sum_{i=1}^{\infty} \lambda_i \funcom{i}(\mu_i \xcom{i}) -\sum_{i=1}^{\infty} \lambda_i\funcom{i}^\star \geq \lambda_i(\funcom{i}(\mu_i \xcom{i}) -\funcom{i}^\star) = \lambda_i\epsilon(\mu_i\xcom{i},\funcom{i})$ for any $i$. Thus \eqref{eq:error-metric} holds with $q=0$. Therefore, \Cref{assm:error-metric} holds for $\funclscvx$ with $\epsilon(x,f) = f(x)-f^\star$.
\end{example}

\subsection{Asymptotic lower bounds for zero-respecting algorithms}\label{sec:asymptotic-lower-bounds-zero-resepecting}
This subsection converts a series of finite-dimensional nonasymptotic lower bounds for zero-respecting sequences into an asymptotic lower bound on a Hilbert space.
In particular, we take a sequence of 
hard functions for a fixed budget (i.e.,
the nonconvex lower bound of \citet[Theorem~1]{carmon2020lower} and the convex lower bound of \citet[Theorem~3]{arjevani2019oracle}) 
Using an infinite weighted sum with decaying weights, we construct a single hard function that forms the basis for our asymptotic lower bound.

\begin{assumption}[Nonasymptotic lower bound: zero-respecting algorithms]
    \label{assm:general nonasymptotic}
    Let $\funcls_{d}$ for any $d \in \NN \cup \{ \infty \}$ denote 
a $d$-dimensional function class, let $\funcls_{} := \cup_{d=1}^\infty \funcls_{d}$ with error metric $\epsilon$, and let $\algorder \in \NN$.
There exists a decreasing bijection $\nonasymFunc : [1,\infty) \rightarrow (0,1]$ such that for any budget $\budget \in \NN$, there exists a hard instance $f\in \funcls_{\dimFunc_\budget}$ such that
for any $\algorder$th order zero-respecting sequence, $\{\xt\}_{t=1}^\infty$, with $x^{(1)} = \mathbf{0}$, we have 
    \begin{equation}
        \label{eq:general nonasmyptotic}
        \epsilon(x^{(\budget+1)},f) \geq C(\funcls_{}) \nonasymFunc(\budget),
    \end{equation}
where $C(\funcls_{})>0$ is a problem-independent numeric constant.
\end{assumption}

\begin{example}[Nonconvex hard instances]\label{example-ncvx-class-nonasym}
    For function class $\funclsncvx$ and error metric $\epsilon(x,f) =\|\grad f(x)\|$, \citet[Theorem~1]{carmon2020lower} show that Assumption~\ref{assm:general nonasymptotic} holds with $\algorder=p$ and $\nonasymFunc(\budget) = \budget^{-\frac{p}{p+1}}$. The hard function constructed by \citet{carmon2020lower} is infinitely differentiable, and \citet[Equation~(10) and Lemma~1]{carmon2020lower} show that at the origin only the first chain term contributes to its derivatives. Thus, its derivatives through order $\algorder$ are bounded independently of the chain length, and a fixed rescaling makes their norms and $L_p$ at most $1$ while preserving the stated rate and the function-value bounds. Hence, the hard function belongs to $\funclsncvx$.
\end{example}

\begin{example}[Convex hard instances]\label{example-cvx-class-nonasym}
    For function class $\funclscvx$ and error metric $\epsilon(x,f) =f(x)-f^\star$, \citet[Theorem~3]{arjevani2019oracle} show that Assumption~\ref{assm:general nonasymptotic} holds with $\algorder = p$ and $\nonasymFunc(\budget) = \budget^{-\frac{3p+1}{2}}$. The hard function constructed by \citet[Section~6]{arjevani2019oracle} is $p$ times differentiable with a Lipschitz-continuous $p$th derivative. Setting $\mu_p=D=1$, the construction satisfies $L_p\leq1$, $\grad^r f(\mathbf{0})=0$ for $2\leq r\leq p$, and $\|\grad f(\mathbf{0})\|\leq1$; together with $f(\mathbf{0})=0$, $\dist{\mathbf{0}}{\argmin f}\leq1$, and convexity, this shows that the hard function belongs to $\funclscvx$. Moreover, the construction of \citet[Section~6]{arjevani2019oracle} is a $p$th-order zero-chain, hence the same instance is hard for every $p$th-order zero-respecting sequence.
\end{example}

Then we can convert the $C(\funcls_{})\nonasymFunc(\budget)$ nonasymptotic lower bound to the following $o(\nonasymFunc(\budget))$ asymptotic lower bound.

\begin{theorem}[Asymptotic conversion: zero-respecting algorithms]
\label{thm.lower-bound-fixed-function-convex}
Suppose Assumption~\ref{assm:error-metric} and Assumption~\ref{assm:general nonasymptotic} hold, and let $\littleFunc : [1,\infty) \rightarrow (0,1]$ be a decreasing bijection.
Then,
there exists a function $F\in \funcls_{\infty}$ and a sequence $\budget_1, \budget_2, \dots$ with $\lim_{i \rightarrow \infty} \budget_i = \infty$, such that
\[
\inf_{\alg \in \algclszr^{(\algorder)}}\epsilon_{\budget_i}(\alg,F) \geq \littleFunc(\budget_i)\nonasymFunc(\budget_i)  \quad  \forall i \in \NN.
\]
\end{theorem}

\begin{proof}
Let $q$, $\{\lambda_i\}_{i=1}^\infty$ and $\{\mu_i\}_{i=1}^\infty$ be as per Assumption~\ref{assm:error-metric}
and define $\budget_0 := 0$ and, for $i \ge 1$,
\begin{equation}\label{eq:hilbert-budget-sequence}
\budget_i :=
\max\left\{
\left\lceil \littleFunc^{-1}\!\left(\min\{1, \lambda_i \mu_i^q C(\funcls_{}) \}\right)\right\rceil,\
\budget_{i-1}+1
\right\}.
\end{equation}
The $\budget_{i-1}+1$ term in \Cref{eq:hilbert-budget-sequence} implies $\budget_i \rightarrow \infty$, and since $\littleFunc$ is a decreasing bijection, \Cref{eq:hilbert-budget-sequence} also implies that
\begin{flalign}\label{eq:implied-rho-C-relationship}
\lambda_i \mu_i^{q} C(\funcls_{}) \ge \littleFunc(\budget_i).
\end{flalign}
By \Cref{assm:general nonasymptotic} for each $i$ there exists a positive integer $d_i$ and
function $f_i \in \funcls_{d_i}$ such that
\begin{equation}
    \label{eq:pick-nonasym-hardfunction-convex}
    \epsilon(z_{i}^{(\budget_i+1)},\funcom{i})\geq C(\funcls_{})\nonasymFunc(\budget_i),
\end{equation}
where $\{z_{i}^{(t)}\}_{t=1}^\infty$ is any $\algorder$th order zero-respecting sequence with respect to $\funcom{i}$ with $z_{i}^{(1)}=\mathbf{0}$.
Moreover, by Assumption~\ref{assm:error-metric} we have $F\in \funcls_{\infty}$ where
\[
F(x)=\sum_{i=1}^{\infty}\lambda_i\funcom{i}(\mu_i \xcom{i}): \funcom{i} \in \funcls_{d_i}, x = (\xcom{i})_{i\ge 1}.
\]
For any $i\in\NN$, consider the suboptimality of any $\alg\in \algclszr^{(\algorder)}$, let $\{\xt\}_{t=1}^\infty = \alg[F]$, we have
\begin{equation*}
        \epsilon_{\budget_i}(\alg, F) = \epsilon(x^{(\budget_i+1)},F)\geq_{(a)}\lambda_i\mu_i^q\epsilon(\mu_i\xcom{i}^{(\budget_i+1)},\funcom{i})\geq_{(b)}\lambda_i\mu_i^q C(\funcls_{})\nonasymFunc(\budget_i)\geq_{(c)}\littleFunc(\budget_i)\nonasymFunc(\budget_i),
\end{equation*}
where $(a)$ is from \eqref{eq:error-metric}, $(b)$ is from the fact that if $\{\xcom{i}^{(t)}\}_{t=1}^\infty$ is a $\algorder$th order zero-respecting sequence with respect to $\lambda_i\funcom{i}(\mu_i\cdot)$ then $\{\mu_i\xcom{i}^{(t)}\}_{t=1}^\infty$ is a $\algorder$th order zero-respecting sequence with respect to $\funcom{i}$ and \eqref{eq:pick-nonasym-hardfunction-convex}~(see the discussion as in \citep[Observation 2]{carmon2020lower}), and $(c)$ is from \eqref{eq:implied-rho-C-relationship}.
\end{proof}

\begin{corollary}[Nonconvex lower bound: zero-respecting algorithms]
    \label{coro:nonconvex-asym-lb}
    Let $\littleFunc : [1,\infty) \rightarrow (0,1]$ be a decreasing bijection.
    Then, for any $p\in \NN$ there exists a function $F\in \funclsncvxhil$ such that
    \[
    \limsup_{\budget \to \infty} \inf_{\alg\in {\algclszr^{(p)}}} \frac{\budget^{\frac{p}{p+1}}\|\grad F(\alg[F]_{\budget})\|}{\littleFunc(\budget)} \geq 1,
    \]
    where $\alg[F]_\budget=x^{(\budget+1)}$ denotes the output after $\budget$ oracle calls.
\end{corollary}
\begin{proof}
    This is a direct result of \Cref{thm.lower-bound-fixed-function-convex}, \Cref{example-ncvx-class} and \Cref{example-ncvx-class-nonasym}.   
\end{proof}

\begin{corollary}[Convex lower bound: zero-respecting algorithms]
    \label{coro:convex-asym-lb}
    Let $\littleFunc : [1,\infty) \rightarrow (0,1]$ be a decreasing bijection.
    Then, for any $p\in \NN$, there exists a function $F\in \funclscvxhil$ such that
    \[
    \limsup_{\budget \to \infty} \inf_{\alg\in \algclszr^{(p)}} \frac{\budget^{\frac{3p+1}{2}}\left(F(\alg[F]_{\budget})-F^\star\right)}{\littleFunc(\budget)} \geq 1,
    \]
    where $\alg[F]_\budget=x^{(\budget+1)}$ denotes the output after $\budget$ oracle calls.
\end{corollary}
\begin{proof}
    This is a direct result of \Cref{thm.lower-bound-fixed-function-convex}, \Cref{example-cvx-class} and \Cref{example-cvx-class-nonasym}.
\end{proof}

\subsection{Asymptotic lower bounds for randomized algorithms}
\Cref{sec:asymptotic-lower-bounds-zero-resepecting} provides 
a lower bound for zero-respecting methods (which, by a black-box reduction,
implies lower bounds for deterministic methods; see \citet[Section 3.3]{carmon2020lower}).
This subsection gives the same result but for randomized methods. 
Overall our strategy is very similar: we construct a single hard distribution from a series 
of hard distributions, each targeting a particular iteration budget; the main difference is that
we now need to construct a hard \emph{distribution} over functions.
The need to construct a distribution over functions makes the proof a little more technically challenging than the proof in \Cref{sec:asymptotic-lower-bounds-zero-resepecting} but the ideas remain similar.

\begin{assumption}[Nonasymptotic lower bound: randomized algorithms]
    \label{assm:general nonasymptotic rand}
    Let $\funcls_{d}$ for any $d \in \NN \cup \{ \infty \}$ denote 
a $d$-dimensional function class, let $\funcls_{} := \cup_{d=1}^\infty \funcls_{d}$ with error metric $\epsilon$, and let $\algorder\in \NN$ such that the $\algorder$th-order oracle in \Cref{def:randomized-alg} is well defined on every $\funcls_d$. There exist a decreasing bijection $\nonasymFunc : [1,\infty) \rightarrow (0,1]$ and a constant $C(\funcls_{})>0$ such that for any budget $\budget\in\NN$ and $\delta\in (0,\frac{1}{2}]$, there exist $\dimFunc_{\budget,\delta}\in\NN$ and a probability measure $\probmeasure_{\budget,\delta}$ on functions in $\funcls_{\dimFunc_{\budget,\delta}}$ such that for any $\alg \in \algclsrand^{(\algorder)}$, we have
    \begin{equation}
        \label{eq:general nonasmyptotic rand}
        \pr_{\alg,f\sim \probmeasure_{\budget,\delta}}\left(\epsilon(x^{(\budget+1)},f) \geq C(\funcls_{}) \nonasymFunc(\budget)\right) \geq 1-2\delta,
    \end{equation}
where $\{\xt\}_{t=1}^\infty$ is the realization of $\alg[\xi,f]$, and the random seed $\xi$ is independent of $f$.
\end{assumption}

\begin{example}[Randomized nonconvex hard instances]
    \label{example-ncvx-class-nonasym-rand}
    For function class $\funclsncvx$ and error metric $\epsilon(x,f)=\|\grad f(x)\|$, \citet[Equation~(16)]{carmon2020lower} imply that Assumption~\ref{assm:general nonasymptotic rand} holds with $\algorder=p$ and $\nonasymFunc(\budget) = \budget^{-\frac{p}{p+1}}$. The randomized construction of \citet[Equation~(13) and Lemma~6]{carmon2020lower} is infinitely differentiable and has chain-length-independent bounds on its derivatives through order $p$ at the origin after the same fixed normalization as above. Hence, its distribution is supported on $\funclsncvx$ without changing the stated rate.
\end{example}

\begin{theorem}[Asymptotic conversion: randomized algorithms]
\label{thm.lower-bound-fixed-function-convex-rand}
Suppose \Cref{assm:error-metric} and \Cref{assm:general nonasymptotic rand} hold, and let $\littleFunc : [1,\infty) \rightarrow (0,1]$ be a decreasing bijection.
Then,
there exists a sequence $\budget_1, \budget_2, \dots$ with $\lim_{i \rightarrow \infty} \budget_i = \infty$, such that for any $\delta\in (0,\frac{1}{2}]$, there exists a probability measure $\probmeasure$ on functions in $\funcls_{\infty}$ such that for any $\alg\in \algclsrand^{(\algorder)}$, we have
\begin{equation}\label{eq:high-prob-rand-lb}
\pr_{\alg, F\sim \probmeasure}\left(\epsilon(x^{(\budget_i+1)},F)\geq \littleFunc(\budget_i)\nonasymFunc(\budget_i), \ \forall i\in \NN \right)\geq 1-2\delta,
\end{equation}
where $\{\xt\}_{t=1}^\infty$ is the realization of $\alg[\xi,F]$.
\end{theorem}
\begin{proof}
Let $q$, $\{\lambda_i\}_{i=1}^\infty$ and $\{\mu_i\}_{i=1}^\infty$ be as per Assumption~\ref{assm:error-metric}. 
Define $\budget_i$ as in
\eqref{eq:hilbert-budget-sequence}; in particular, this sequence does not
depend on $\delta$. Fix $\delta\in(0,\frac12]$ and set
$\delta_i = \frac{\delta}{2^{i}}$. 
By \Cref{assm:general nonasymptotic rand} for each $i$ there exists a positive integer $d_i$ and
distribution $\probmeasure_i$ on $\mathcal{F}_{d_i}$ such that
for any $\widetilde{\alg}\in \algclsrand^{(\algorder)}$,
\begin{equation}\label{eq:randomized-f_i-is-hard}
    \pr_{\widetilde{\alg},\funcom{i}\sim \probmeasure_i}
    \left(
        \epsilon(z_i^{(\budget_i+1)},\funcom{i})
        <
        C(\funcls_{}) \nonasymFunc(\budget_i)
    \right)
    \leq 2\delta_i
\end{equation}
where
$\{z_i^{(t)}\}_{t=1}^\infty = \widetilde{\alg}[\xi,\funcom{i}]$.
Consider the following random function
$F\in \funcls_\infty$:
\[
F(x)=\sum_{i=1}^{\infty}\lambda_i\funcom{i}(\mu_i \xcom{i}),
\qquad
\funcom{i} \sim \probmeasure_i,
\qquad
x = (\xcom{i})_{i\ge 1},
\]
where the functions $\{\funcom{i}\}_{i\geq 1}$ are sampled
independently.
Denote the induced probability measure over $\funcls_\infty$ as
$\probmeasure$.
Fix $\alg\in\algclsrand^{(\algorder)}$. Let $\mathbb P$ denote the joint probability measure of the random
seed $\xi$ and the independent random functions
$\{\funcom{j}\}$. Under $\mathbb P$, the marginal distribution of $F$ is
$\probmeasure$, and $F$ is independent of the algorithm's
random seed $\xi$. Thus,
probabilities involving $F$ and the iterates of $\alg$ are the same under
$\mathbb P$ and $\pr_{\alg,F\sim\probmeasure}$. 

Next, fix $i$ and a realization of the other functions
$(\funcom{j})_{j\neq i}$, denoted by $\funcom{-i}$. 
Since $\xt$, $\funcom{-i}$ and $\grad^{(0,\ldots,\algorder)}\funcom{i}(\mu_i\xcom{i}^{(t)})$ can 
be used to determine $\grad^{(0,\ldots,\algorder)}F(\xt)$,
we can unroll $\alg^{(t)}$
to obtain a randomized algorithm $\underline{\alg}_{i,\funcom{-i}}^{(t)}$ such that
\begin{equation}\label{eq:block-equivalent-alg}
\mu_i \xcom{i}^{(t)} = \underline{\alg}_{i,\funcom{-i}}^{(t)}\left(\xi, \grad^{(0,\ldots,\algorder)} \funcom{i}(\mu_i \xcom{i}^{(1)}),\ldots, \grad^{(0,\ldots,\algorder)} \funcom{i}(\mu_i \xcom{i}^{(t-1)})\right)
\text{ for } t\in \NN .
\end{equation}
To analyze this sequence $\{\mu_i \xcom{i}^{(t)}\}_{t=1}^{\infty}$ we introduce two terms:
$h_i(\xi, \funcom{i}, \funcom{-i}) := \mathbf{1}(\epsilon(\mu_i\xcom{i}^{(\budget_i+1)},\funcom{i}) < C(\funcls_{})\nonasymFunc(\budget_i))$
which is the indicator that the error falls below the threshold $C(\funcls_{})\nonasymFunc(\budget_i)$,
and the associated probability $p_i(\funcom{-i})
:=
\pr_{\xi \sim \mathsf{Unif}([0,1]),\,\funcom{i}\sim\probmeasure_i}( h_i(\xi, \funcom{i}, \funcom{-i}) = 1 )$. 
Then, by $(a)$ the law of total expectation, $(b)$ mutual independence of $\xi$,
$\funcom{i}$ and $\funcom{-i}$, and $(c)$ 
\Cref{eq:randomized-f_i-is-hard} applied to
$\widetilde{\alg}=\underline{\alg}_{i,\funcom{-i}}$ with
$z_i^{(t)}=\mu_i\xcom{i}^{(t)}$ we get
\[
\mathbb P
(
    h_i(\xi, \funcom{i}, \funcom{-i}) = 1
) = \E\left[ h_i(\xi, \funcom{i}, \funcom{-i}) \right] \underset{(a)}{=}
\E\left[
\E\left[ h_i(\xi, \funcom{i}, \funcom{-i}) \mid \funcom{-i} \right] \right]
\underset{(b)}{=}
\E\left[ p_i(\funcom{-i}) \right]
\underset{(c)}{\le}
2\delta_i.
\]
Applying (a) \Cref{eq:error-metric}, 
(b) \Cref{eq:implied-rho-C-relationship},
(c) a union bound, and
(d) the previous inequality:
\[
\begin{aligned}
&
\mathbb P
(
    \forall i\in\NN:\ 
    \epsilon(x^{(\budget_i+1)},F)
    \geq
    \littleFunc(\budget_i)\nonasymFunc(\budget_i)
)
\\
&\underset{(a)}{\geq}
\mathbb P
(
    \forall i\in\NN:\ 
    \lambda_i\mu_i^q
    \epsilon(\mu_i\xcom{i}^{(\budget_i+1)},\funcom{i})
    \geq
    \littleFunc(\budget_i)\nonasymFunc(\budget_i)
) \underset{(b)}{\ge} \mathbb P
( \forall i\in\NN:\
    \epsilon
    (
        \mu_i\xcom{i}^{(\budget_i+1)},
        \funcom{i}
    )
    \ge
    C(\funcls_{})\nonasymFunc(\budget_i)
)
\\
&\underset{(c)}{\geq}
1-
\sum_{i=1}^{\infty}
\mathbb P
(
    \epsilon
    (
        \mu_i\xcom{i}^{(\budget_i+1)},
        \funcom{i}
    )
    <
    C(\funcls_{})\nonasymFunc(\budget_i)
)
\underset{(d)}{\geq} 1 - 2 \sum_{i=1}^{\infty} \delta_i = 1 - \delta \sum_{i=1}^{\infty} 2^{1-i}
=1-2\delta.
\end{aligned}
\]
\end{proof}

\begin{corollary}[Nonconvex lower bound: randomized algorithms]
    \label{coro:informal-randomized}
    Let $\littleFunc : [1,\infty) \rightarrow (0,1]$ be a decreasing bijection. Then, for any $\delta\in (0,\frac{1}{2}]$ and $p\in \NN$, there exists a probability measure $\probmeasure_\delta$ on the functions in $\funclsncvxhil$ such that for any $\alg \in \algclsrand^{(p)}$
    \begin{equation}
        \label{eq:informal-randomized}
        \pr_{\alg,F\sim \probmeasure_\delta}\left(\limsup_{\budget \to \infty} \frac{\budget^{\frac{p}{p+1}}\|\grad F(\alg[F]_{\budget})\|}{\littleFunc(\budget)} \geq 1 \right)\geq 1-2\delta.
    \end{equation}
\end{corollary}
\begin{proof}
    This is a direct result of \Cref{thm.lower-bound-fixed-function-convex-rand}, \Cref{example-ncvx-class} and \Cref{example-ncvx-class-nonasym-rand}.
\end{proof}
Our general technique can also be applied to the convex lower bounds for randomized first-order methods of \citet[Theorem~1]{woodworth2018graph}, with a minor modification to their theorem statement and proof. We provide this proof in \Cref{app:convex-fixed-rand}.
\begin{corollary}[Convex lower bounds: randomized first-order algorithms]
\label{coro:convex-fixed-rand}
Let
\[
\mathcal{C}^{\text{cvx,bd}}_\infty
:=
\left\{
f:\hilbert{\infty}\to\mathbb{R}
:
f \text{ is convex},\;
\dist{\mathbf{0}}{\argmin f}\le 1,\;
-1\le f^\star\le f(\mathbf{0})\le 1
\right\}
\]
denote the class of convex functions whose minimizer sets have distance at most one from the origin and whose optimal value and objective value at the origin are bounded.
Let $\littleFunc : [1,\infty)\to (0,1]$ be a decreasing bijection. Then, for
any $\delta\in(0,\frac12]$, the following two statements hold.
\begin{enumerate}[i.]
    \item \textbf{Nonsmooth case.}
    There exists a probability measure
    $\probmeasure_{\delta}^{\rm ns}$ on the set of nonsmooth convex functions, $\mathcal{F}^{\rm ns}_\infty:=\mathcal{C}^{\text{cvx,bd}}_\infty\cap \lipsfunc{\hilbert{\infty}}{0}$,
    such that, for any randomized
    first-order algorithm $\alg\in\algclsrand^{(1)}$,
    \[
    \pr_{\alg,\,F\sim\probmeasure_{\delta}^{\rm ns}}
    \left(
    \limsup_{\budget\to\infty}
    \frac{\sqrt{\budget}\left(F(\alg[F]_{\budget})-F^\star\right)}
    {\littleFunc(\budget)}
    \ge 1
    \right)
    \ge 1-2\delta .
    \]

    \item \textbf{Smooth case.}
    There exists a probability measure
    $\probmeasure_{\delta}^{\rm sm}$ on the set of smooth convex functions,
    $\mathcal{F}^{\rm sm}_\infty:=\mathcal{C}^{\text{cvx,bd}}_\infty \cap \lipsfunc{\hilbert{\infty}}{1}$, such that, for any randomized first-order
    algorithm $\alg\in\algclsrand^{(1)}$,
    \[
    \pr_{\alg,\,F\sim\probmeasure_{\delta}^{\rm sm}}
    \left(
    \limsup_{\budget\to\infty}
    \frac{\budget^2\left(F(\alg[F]_{\budget})-F^\star\right)}
    {\littleFunc(\budget)}
    \ge 1
    \right)
    \ge 1-2\delta .
    \]
\end{enumerate}
\end{corollary}
% \hinder{we really ought to have a corollary here with the convex case}

% \hinder{If you need to you can reprove the results of \citep[Theorems~1 and~2]{woodworth2018graph} 
% in the appendix in the format that we need. If you give 
% ChatGPT the tex of the paper it can probably be very helpful
% to figure out an initial draft. You should 
% copy their proof as much as possible (with minor edits to improve 
% readability) and 
% clearly highlight the changes.}

\section{A method that achieves an \texorpdfstring{$o(\budget^{-1/2})$}{o(n\string^(-1/2))} asymptotic convergence rate for stochastic convex optimization}\label{sec:upper-bounds}
In this section, we provide upper bounds on the asymptotic convergence of stochastic Lipschitz convex optimization problems.
Our method does not require any tuning parameter or knowledge of problem parameters (i.e., it is parameter-free).
We provide a matching lower bound in \Cref{sec.lower bound}.

\paragraph{Stochastic optimization (SO) problems} A SO problem instance is a tuple  $(f, P)$ containing a distribution $P$ over sample space $\mathcal{S}$ and a sample objective $f: \hilbert{d} \times \mathcal{S} \rightarrow \mathbb{R}$, $d\in \NN$. We write
\[
F_{f, P}(x):=\mathbb{E}_{S \sim P} f(x ; S), \quad F_{f, P}^\star := \min_{x \in \hilbert{d}} F_{f, P}(x)  \quad \text{and} \quad X_{f, P}^{\star}:=\underset{x \in \hilbert{d}}{\arg \min}~ F_{f, P}(x)
\]
for the population objective (which we wish to minimize) and its set of minimizers, respectively. 
We omit the subscript $f,P$ when it does not cause confusion.

We consider the fundamental class of SO instances with a minimizer at most $R$ away from the origin\footnote{This implicitly, and without loss of generality, assumes that optimization methods are initialized at $x_1=\mathbf{0}$.} and bounded stochastic gradients.

\begin{assumption}
\label{assume-nonsmooth-convex-optimization}
There exists some $R > 0$ such that $\dist{\mathbf{0}}{X_{f, P}^{\star}} \leq R$ almost surely.
Also, there exists some $L > 0$ such that $f(\cdot;s)$ is $L$-Lipschitz and convex for all $s \in \mathcal{S}$.
\end{assumption}
In particular, we show that its output $\xout{\budget}$  after $\budget$ stochastic gradient evaluations satisfies 
\begin{equation*}
    \limsup_{\budget\to \infty} \budget^{1/2}\mathbb{E}\left[F(\xout{\budget})-F^\star\right]=0.
\end{equation*}
Thus our algorithm has an asymptotic error bound of $o(\budget^{-1/2})$. Our method has two distinct stages. The first stage estimates the initial distance to optimality. In the second stage, \Cref{alg:nested-adaptive} uses the adaptive gradient method to optimize the function over a sequence of shrinking balls: each time the ball shrinks and the minimizer is inside the new ball, the algorithm improves its worst-case convergence bound due to the smaller search space.

Critical to proving our result is \Cref{lem:monotone-error-bound}, which shows that all convex optimization problems in finite-dimensional Hilbert spaces with continuous objectives satisfy an error bound condition.
\Cref{lem:monotone-error-bound} allows us to argue that our method gets closer and closer to the optimal solution, enabling a faster decay of the suboptimality gap.
\Cref{lem:monotone-error-bound} is similar to the \L{}ojasiewicz inequality~\citep{lojasiewicz1963propriete}. However, the \L{}ojasiewicz inequality applies to real analytic functions, whereas \Cref{lem:monotone-error-bound} applies to continuous convex functions.
The proof of \Cref{lem:monotone-error-bound} uses elementary convex analysis techniques (\Cref{lem:monotone-error-bound} is not
particularly novel; for example, see \citet{nesterov2025universal} for a similar result).
Similar to the \L{}ojasiewicz inequality, the proof depends on the compactness of $\mathcal{X}$, and it is straightforward
to show this result fails for infinite dimensional spaces even when $\mathcal{X}$ is closed and 
bounded\footnote{Let $\mathcal{X}:=\{x\in \hilbert{\infty} : \|x\| \le 1\}$ and $\hat F(x):=\sum_{i=1}^\infty \frac{1}{i}\,[x]_i^2$ 
where $[x]_i$ denotes the $i$th component of $x$. Clearly, $\hat F$ is continuous and convex on $\mathcal{X}$. 
Moreover, $H^\star=\{0\}$ and $\hat F^\star=0$.
Let $e_n$ be a standard basis vector, then $\dist{e_n}{H^\star}=\|e_n\|_2=1$ while 
$\hat F(e_n)-\hat F^\star=\hat F(e_n)=\frac{1}{n}\to 0$.
If Lemma~\ref{lem:monotone-error-bound} held then $\phi(1) \le \lim_{n \rightarrow \infty}\hat{F}(e_n) - \hat{F}^\star = 0$ 
but also $\phi(0)=0$ which contradicts $\phi$ being a strictly increasing function.
}.

\begin{lemma}\label{lem:monotone-error-bound}
Suppose that $\hat{F} : \mathcal{X} \rightarrow \R$ is a continuous convex function and $\mathcal{X}\subseteq \hilbert{d}$ is a compact convex set. 
Let $D = \max_{x \in \mathcal{X}} \dist{x}{H^\star}$.
Then, there exists a strictly increasing function $\phi : [0,D] \rightarrow [0,\infty)$ such that for all $x \in \mathcal{X}$,
\[
\phi(\dist{x}{H^\star}) \le \hat{F}(x) - \hat{F}^\star \text{ where } H^\star := \argmin_{x \in \mathcal{X}} \hat{F}(x) \text{ and } \hat{F}^\star := \min_{u \in \mathcal{X}} \hat{F}(u).
\]
\end{lemma}
\begin{proof}
Define:
\[
\phi( t ) := \min_{x\in S(t)} \hat{F}(x) - \hat{F}^\star \text{ where } S(t) := \{ x \in \mathcal{X} : \dist{x}{H^\star} \ge t \}.
\]
Assume $\mathcal{X}$ is nonempty otherwise the result trivially holds.
Before beginning the proof, we confirm a few trivial formalities: (i) $H^{\star}$ is convex, compact, and nonempty because $\mathcal{X}$ is a convex compact nonempty set and $\hat{F}$ is a continuous convex function, and (ii)
$S(t)$ is compact because $\mathcal{X}$ is compact and $\dist{x}{H^\star}$ is a continuous function of $x$.

The crux of the proof is to show that $\phi$ is strictly increasing.
Suppose $0 \le t' < t \le D$. Let $u \in \argmin_{x\in S(t)} \hat{F}(x)$, $x^\star \in H^\star$, and
$w(\lambda) := (1 - \lambda) u + \lambda x^\star$ where $w : [0,1] \rightarrow \mathcal{X}$.
Since $h(\lambda) := \dist{w(\lambda)}{H^\star}$ is a continuous function with $h(0) \ge t > 0$ and $h(1) = 0$, by the intermediate value theorem there exists some $\lambda' \in (0,1]$ such that $\dist{w(\lambda')}{H^\star} = h(\lambda') = t'$. Moreover, since $x^\star, u \in \mathcal{X}$ and $\mathcal{X}$ is convex we have $w(\lambda') \in \mathcal{X}$. We conclude that $w(\lambda') \in S(t')$.
Then, by (a) $\lambda' > 0$ and $\hat{F}(u) > \hat{F}(x^\star)$ because $\dist{u}{H^\star} \ge t > 0$, (b) convexity, and (c) $w \in S(t')$ we have 
\[
\min_{x \in S(t)} \hat{F}(x) = \hat{F}(u) >_{(a)}  (1 - \lambda' ) \hat{F}(u) + \lambda' \hat{F}(x^\star) \ge_{(b)} \hat{F}(w(\lambda')) \ge_{(c)} \min_{x \in S(t')} \hat{F}(x).
\]
Thus, $\phi(t) = \min_{x \in S(t)} \hat{F}(x) - \hat{F}^\star > \min_{x \in S(t')} \hat{F}(x) - \hat{F}^\star = \phi(t')$ as desired.
\end{proof}
\subsection{Stage I: estimating the initial distance to optimality}

For any given stochastic gradient evaluation budget $\budget>0$, we allocate $\Fbudget = \lfloor \budget /2 \rfloor$ of it to the first stage where we estimate the initial distance to optimality. In particular, the first stage considers the following regularized problem:
\begin{equation}
    \label{eq.auxiliary problem in stage 1 high prob}
    f_\Fbudget(x; s) := f(x; s)+\Fbudget^{-\alpha}\|x\|, \ 0<\alpha<\frac{1}{4}.
\end{equation}
and we apply SGD to \eqref{eq.auxiliary problem in stage 1 high prob} for $\Fbudget$ iterations,
\begin{equation}
    \label{alg:sgd high prob}
    \xsgd{t+1} \gets \xsgd{t} - \eta \gsgd
\end{equation}
where $\gsgd\in\partial f_\Fbudget(\xsgd{t};S_t)$, $S_t\stackrel{\mathrm{iid}}{\sim}P$,
starting from $\xsgd{1}=\mathbf{0}$ with step size $\eta = \Fbudget^{-3/4}$
and output $\xB=\tfrac{1}{\Fbudget}\sum_{t=1}^\Fbudget \xsgd{t}$. 
%We establish the error bound results of $\xB$.
\begin{lemma}\label{lem:iterate-bound}
For SGD (i.e., \eqref{alg:sgd high prob}) with a step size $\eta = \Fbudget^{-3/4}$, 
for all $t \le \Fbudget$ and $\Fbudget^{-\alpha} \leq L$, we have $\| \xsgd{t} \| \le 2 L \Fbudget^{1/4}$.
\end{lemma}

\begin{proof}
Observe that by $(a)$ the triangle inequality, $(b)$ induction and $t\leq \Fbudget$, $x_1 = \mathbf{0}$, and $(c)$ the assumption $\Fbudget \geq L^{-1/\alpha}$, we get
\begin{equation*}
\| \xsgd{t} \| = \|\xsgd{t-1} -\eta g_{t-1}\|\leq_{(a)} \|\xsgd{t-1}\| + \Fbudget^{-3/4}(L+\Fbudget^{-\alpha})\leq_{(b)}\Fbudget^{1/4}(L+\Fbudget^{-\alpha}) \leq_{(c)} 2L\Fbudget^{1/4}.
\end{equation*}
\end{proof}
\Cref{lem.sgd error bound analysis high prob} is an immediate consequence of a standard result of \citet{orabona2025modern}. It shows that the suboptimality decays with high probability and that the regularizer guarantees that the norm of SGD's output, $\| \xB \|$, is not too big.

\begin{lemma}
\label{lem.sgd error bound analysis high prob}
Suppose \Cref{assume-nonsmooth-convex-optimization} holds. Then, for any $\alpha \in (0,1/4)$, $\Fbudget^{-\alpha} \le L$, $x \in \ball{2 L \Fbudget^{1/4}}$, and $\delta \in (0,1)$, with probability at least $1-\delta$,
\[
F(\xB) - F(x) \le \frac{\| x \| - \| \xB \|}{\Fbudget^{\alpha}} + \frac{ \| x \|^2 + 36 L^2 \sqrt{2 \log \frac{2}{\delta}}}{2 \Fbudget^{1/4} }.
\]
\end{lemma}
\begin{proof}
We intend to apply the result of \citet[Theorem~3.10]{orabona2025modern}, but this requires a mapping from $\ball{2 L \Fbudget^{1/4}} \rightarrow [0,1]$. To achieve this, consider the following rescaling of $f_\Fbudget$,
\[
\tilde{f}_\Fbudget(x;S) := \frac{f_\Fbudget(x;S) - f_\Fbudget(\mathbf{0};S) + 4 L^2 \Fbudget^{1/4}}{8 L^2 \Fbudget^{1/4}} \quad \forall x \in \ball{2 L \Fbudget^{1/4}}.
\]
Since $\abs{f_\Fbudget(x;S) - f_\Fbudget(\mathbf{0};S)} \le (L+\Fbudget^{-\alpha}) \|x\| \le 4 L^2 \Fbudget^{1/4}$ we deduce that $\tilde{f}_\Fbudget(x;S) \in [0,1]$.
Since $\xsgd{t} \in \ball{2 L \Fbudget^{1/4}}$ from \Cref{lem:iterate-bound} for any realization, the result of \citet[Theorem~3.10]{orabona2025modern}, applied to $\tilde{f}_\Fbudget$, together with Jensen's inequality applied to $F_\Fbudget$, gives
\[
F_\Fbudget(\xB) - F_\Fbudget(x) \le \frac{R(x, \Fbudget)}{\Fbudget} + 16 L^2 \Fbudget^{1/4} \sqrt{\frac{2 \log \frac{2}{\delta}}{\Fbudget}}
\]
where $R(x, \Fbudget)$ is the regret of our algorithm after $\Fbudget$ steps. 
Applying the regret bound of \citet[Theorem~2.13]{orabona2025modern} with $x_1 = \mathbf{0}$ and $\eta = \Fbudget^{-3/4}$ to $R(x, \Fbudget)$ gives
\[
F_\Fbudget(\xB) - F_\Fbudget(x) \le \frac{\| x \|^2}{2 \Fbudget^{1/4}} + 2L^2 \Fbudget^{-3/4} + 16 L^2  \Fbudget^{1/4} \sqrt{\frac{2 \log \frac{2}{\delta}}{\Fbudget}}.
\]
Substituting for $F_\Fbudget$ and using $\Fbudget \ge 1$ gives the desired bound.
\end{proof}
\noindent
The goal of stage one is to use $\|\xB\|$ to estimate the initial distance to optimality. 
The event,
\[
\mathcal{E}_\Fbudget = \left\llbracket\frac{1}{2} \dist{\mathbf{0}}{X^\star}\leq \|\xB\| \leq 2R  \right\rrbracket,
\]
measures whether our goal was achieved.
\Cref{lem.xb converge non asym} shows that for sufficiently large $\Fbudget$ the event $\mathcal{E}_\Fbudget$ will occur with high probability. 
The proof focuses on the case in which the suboptimality bound given in \Cref{lem.sgd error bound analysis high prob} holds and shows that this implies that $\| \xB \|$ is, for sufficiently large $\Fbudget$, bounded by $2R$. By additionally employing \Cref{lem:monotone-error-bound} and the fact that the suboptimality tends to zero, we establish $\frac{1}{2} \dist{\mathbf{0}}{X^\star}\leq \|\xB\|$.

\begin{lemma}
\label{lem.xb converge non asym}
Suppose \Cref{assume-nonsmooth-convex-optimization} holds. For any $\delta \in (0,1)$, if 
%\hinder{probably can simplify this down for the final version and keep this bound in the proof in the appendix}
\begin{equation*}
\begin{gathered}
\Fbudget \geq \max\left\{
\frac{1}{L^{\frac{1}{\alpha}}},\,
\frac{R^4}{16L^4},\,
B_\delta,\,
\BudgetLowerBoundPhi
\right\},\\[1ex]
B_\delta
:=
\begin{cases}
\left(\frac{36 L^2 \sqrt{2\log\frac{2}{\delta}}}{4R}\right)^{\frac{1}{1/4-\alpha}} & \text{if } \mathbf{0} \in X^\star \\
\left(\frac{\|\Pi_{X^\star}(\mathbf{0})\|^2+36 L^2 \sqrt{2\log\frac{2}{\delta}}}{2\|\Pi_{X^\star}(\mathbf{0})\|}\right)^{\frac{1}{1/4-\alpha}} & \text{otherwise},
\end{cases},\qquad
\BudgetLowerBoundPhi
:=
\begin{cases}
0 & \text{if } \mathbf{0} \in X^\star \\
\left(\frac{2\|\Pi_{X^\star}(\mathbf{0})\|}{\phi\left(\frac{\|\Pi_{X^\star}(\mathbf{0})\|}{2}\right)}\right)^{\frac{1}{\alpha}} & \text{otherwise},
\end{cases}
\end{gathered}
\end{equation*}
then $\pr(\mathcal{E}_\Fbudget) \ge 1 - \delta$, where $\phi$ is the function obtained by applying \Cref{lem:monotone-error-bound} to $F$ on $\ball{2R}$.
\end{lemma}
\begin{proof}
Let $x^\star = \Pi_{X^\star}(\mathbf{0})$ and thus with $\| x^\star \| \le R$.
From the premise that\[\Fbudget \ge \max\{ L^{-1/\alpha}, (R / (2 L))^4 \}\]and \Cref{lem.sgd error bound analysis high prob},
\begin{flalign}\label{eq:sgd-objective-gap-decays}
F(\xB) - F(x^\star) \le \frac{\| x^\star \| - \| \xB \|}{\Fbudget^{\alpha}} + \frac{ \| x^\star \|^2 + 36 L^2 \sqrt{2 \log \frac{2}{\delta}}}{2 \Fbudget^{1/4} }
\end{flalign}
with probability $1-\delta$.
Thus, it suffices to establish $\mathcal{E}_\Fbudget$ under the assumptions that 
\eqref{eq:sgd-objective-gap-decays} holds.
If $\mathbf{0}\in X^\star$, then $x^\star=\mathbf{0}$ and rearranging \eqref{eq:sgd-objective-gap-decays} and applying $F(\xB)-F(x^\star)\geq 0$ yields
\begin{equation}\label{eq:zero-optimum-stage-one-bound}
\|\xB\|\leq \frac{36 L^2 \sqrt{2\log\frac{2}{\delta}}}{2\Fbudget^{1/4-\alpha}}
\leq 2R,
\end{equation}
where the last inequality uses that $\Fbudget\geq B_\delta$.
Therefore, $\mathcal{E}_\Fbudget$ holds. For the remainder of the proof, suppose $\mathbf{0}\notin X^\star$, and hence $x^\star\neq\mathbf{0}$.

Rearranging \eqref{eq:sgd-objective-gap-decays} and applying $F(\xB) - F(x^\star) \ge 0$ yields
\[
\| \xB \|  \le \| x^\star \| + \frac{ \| x^\star \|^2 + 36 L^2 \sqrt{2 \log \frac{2}{\delta}}}{2 \Fbudget^{1/4 - \alpha} } \le_{(a)} 2\|x^\star\| \le 2R
\]
where $(a)$ uses that $\Fbudget \ge B_\delta$. Thus $\| \xB \| \le 2R$ as desired.

Next, we prove $\frac{1}{2}\dist{\mathbf{0}}{X^\star}\leq \|\xB\|$. Defining the function $\hat{F}(x) : \ball{2 R} \rightarrow \R$ as  $\hat{F}(x) := F(x)$ for all $x \in \ball{2 R}$, 
using $F(\xB) - F(x^\star) = \hat{F}(\xB) - \hat{F}(x^\star) \ge \phi(\dist{X^\star \cap \ball{2 R}}{\xB}) \ge  \phi(\dist{X^\star}{\xB})$ from \Cref{lem:monotone-error-bound} and applying this inequality to \eqref{eq:sgd-objective-gap-decays} yields,
\begin{equation*}
\phi(\dist{X^\star}{\xB}) \le \frac{\| x^\star \| - \| \xB \|}{\Fbudget^{\alpha}} + \frac{ \| x^\star \|^2 + 36 L^2 \sqrt{2 \log \frac{2}{\delta}}}{2 \Fbudget^{1/4} }
\le_{(a)} \frac{2\|x^\star\|}{\Fbudget^\alpha} \le_{(b)} \phi(\| x^\star \| /2).
\end{equation*}
$(a)$ uses that $\Fbudget \ge B_\delta$ and $(b)$ uses that $\Fbudget \ge \left(\frac{2\|x^\star\|}{\phi(\|x^\star\|/2)}\right)^{1/\alpha}$. Thus, \(\| x^\star \| / 2 \ge \dist{X^\star}{\xB} \ge \|\Pi_{X^\star}(\mathbf{0})\| - \| \xB \| =  \| x^\star \| - \| \xB \|\) as desired.
%$\frac{1}{2} \| x^\star \| \le \| \xB \|$.
\end{proof}

\subsection{Stage II: a faster asymptotic rate}

We now use our estimate of the initial distance to optimality to design \Cref{alg:nested-adaptive}, which obtains a $o(\budget^{-1/2})$ convergence rate. Since we allocate $\Fbudget = \lfloor \budget /2 \rfloor$ stochastic gradient evaluations to the first stage, the event that our estimate is sufficiently accurate is $\mathcal{E}_{\lfloor \budget/2 \rfloor}$.

\Cref{alg:nested-adaptive} solves a sequence of subproblems over shrinking balls using Adaptive SGD (also known as isotropic SGD \citep{gupta2017unified}, a scalar variant of \textsc{AdaGrad} \citep{mcmahan2010adaptive,duchi2011adaptive}), each using fewer and fewer iterations. 
Critically, the number of iterations is selected carefully such that, if a minimizer remains inside the ball, the worst-case convergence bound will improve because
$R_k k^2 \rightarrow 0$ as $k \rightarrow \infty$. The random variable $\kn$, which is defined in \Cref{define:kb stochastic}, represents the last subproblem that contains the minimizer. \Cref{fig:overall} provides visualizations to help understand \Cref{alg:nested-adaptive} and \Cref{define:kb stochastic}.

\begin{definition}\label{define:kb stochastic}
For any $\budget > 0$ and $k\leq k_{\max}$, define the event
\begin{equation}\label{eq.define kb stochastic}
\mathcal{B}_{\budget,k}
:=\left\llbracket
    \exists x^\star \in X^\star \ \text{such that} \ x^\star \in\ball{x_{1,k},R_k}
\right\rrbracket.
\end{equation}
On $\mathcal{E}_{\lfloor \budget/2 \rfloor}$, define $\kn$ as the largest index $k\leq k_{\max}$ such that $\mathcal{B}_{\budget,k}$ occurs; on $\mathcal{E}_{\lfloor \budget/2\rfloor}^c$, set $\kn:=0$. In particular, $\kn\geq1$ on $\mathcal{E}_{\lfloor \budget/2 \rfloor}$.
\end{definition}

\begin{algorithm}[H]
\caption{Our parameter-free method with an \texorpdfstring{$o(\budget^{-1/2})$}{o(n\string^(-1/2))} asymptotic convergence rate.}
\label{alg:nested-adaptive}
\begin{algorithmic}[1]
\Require Initial point $x_{1,1} = \mathbf{0}$, budget for total number of stochastic gradient evaluations $\budget>0$
\State $R_1 \gets  2 \| \xStageOne{\lfloor \budget/2 \rfloor} \|$ \Comment{Estimate initial distance to optimality}
\For{$k=1,2,\dots,k_{\max} := \lfloor \sqrt{3 \budget} / \pi \rfloor$}
    \State Set $R_k \gets \tfrac{R_1}{k^3}$, $T_k \gets \left\lfloor \tfrac{\budget}{k^2}\cdot\tfrac{3}{\pi^2}\right\rfloor$
    \For{$t=1,\dots,T_k$}\Comment{Run \ADASGD{} for $T_k$ iterations over $\ball{x_{1,k},R_k}$}
        \State Compute stochastic gradient $g_{t,k}$ at $x_{t,k}$
        \State $x_{t+1,k} \gets \Pi_{\ball{x_{1,k},R_k}}\Biggl(x_{t,k} - \dfrac{R_k}{\sqrt{\sum_{i=1}^t \|g_{i,k}\|^2}}\, g_{t,k}\Biggr)$ 
    \EndFor
    \State Set $x_{1,k+1} \gets \frac{1}{T_k}\sum_{t=1}^{T_k} x_{t,k}$
\EndFor
\State \Return $\xout{\budget} := x_{1,k_{\max}+1}$
\end{algorithmic}
\end{algorithm}
\begin{figure}[ht]
    \centering

    % =======================
    % Top row: three panels
    % =======================
    \begin{minipage}[b]{0.33\textwidth}
        \centering
        \includegraphics[width=\textwidth]{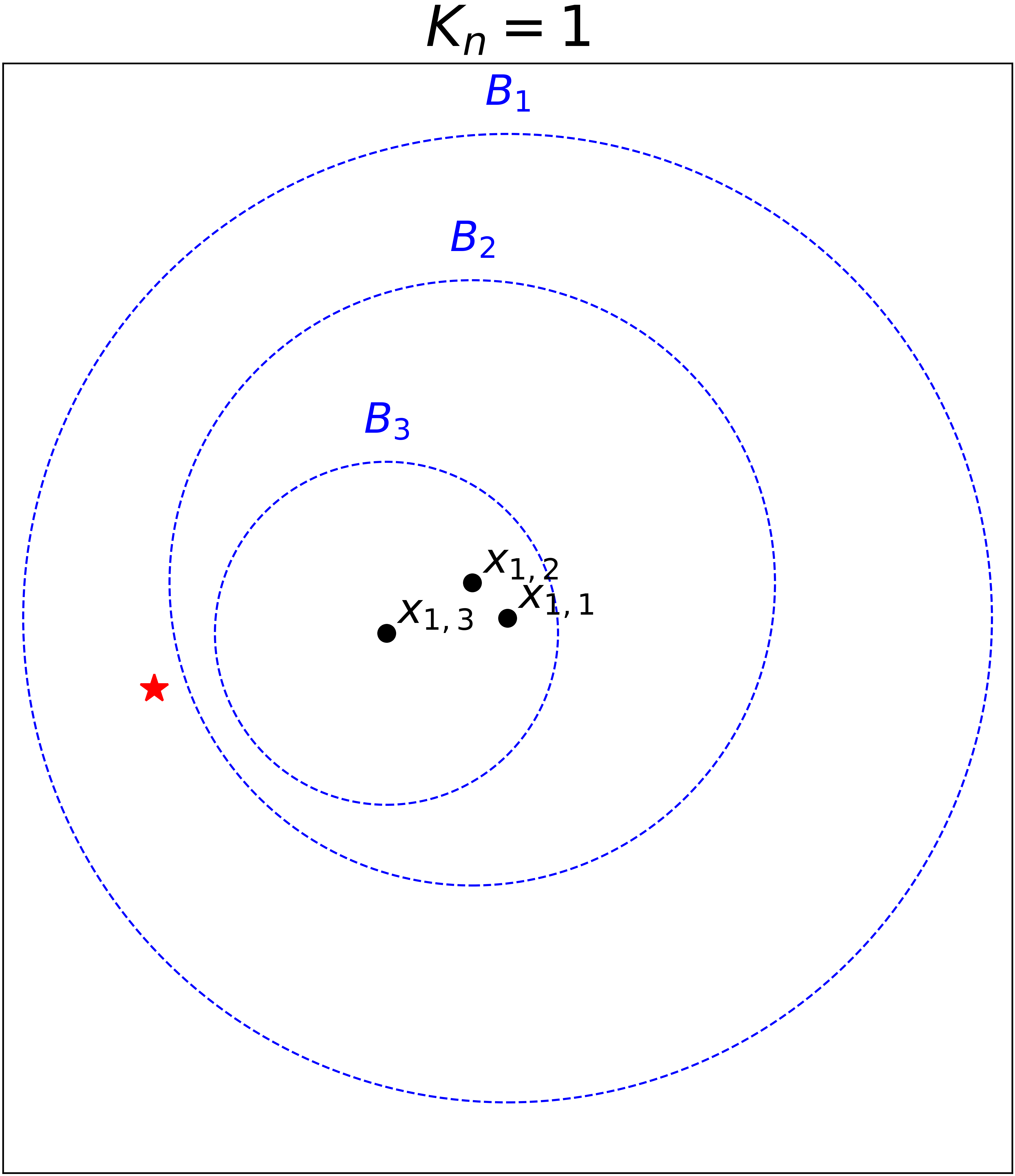}
    \end{minipage}%
    \begin{minipage}[b]{0.33\textwidth}
        \centering
        \includegraphics[width=\textwidth]{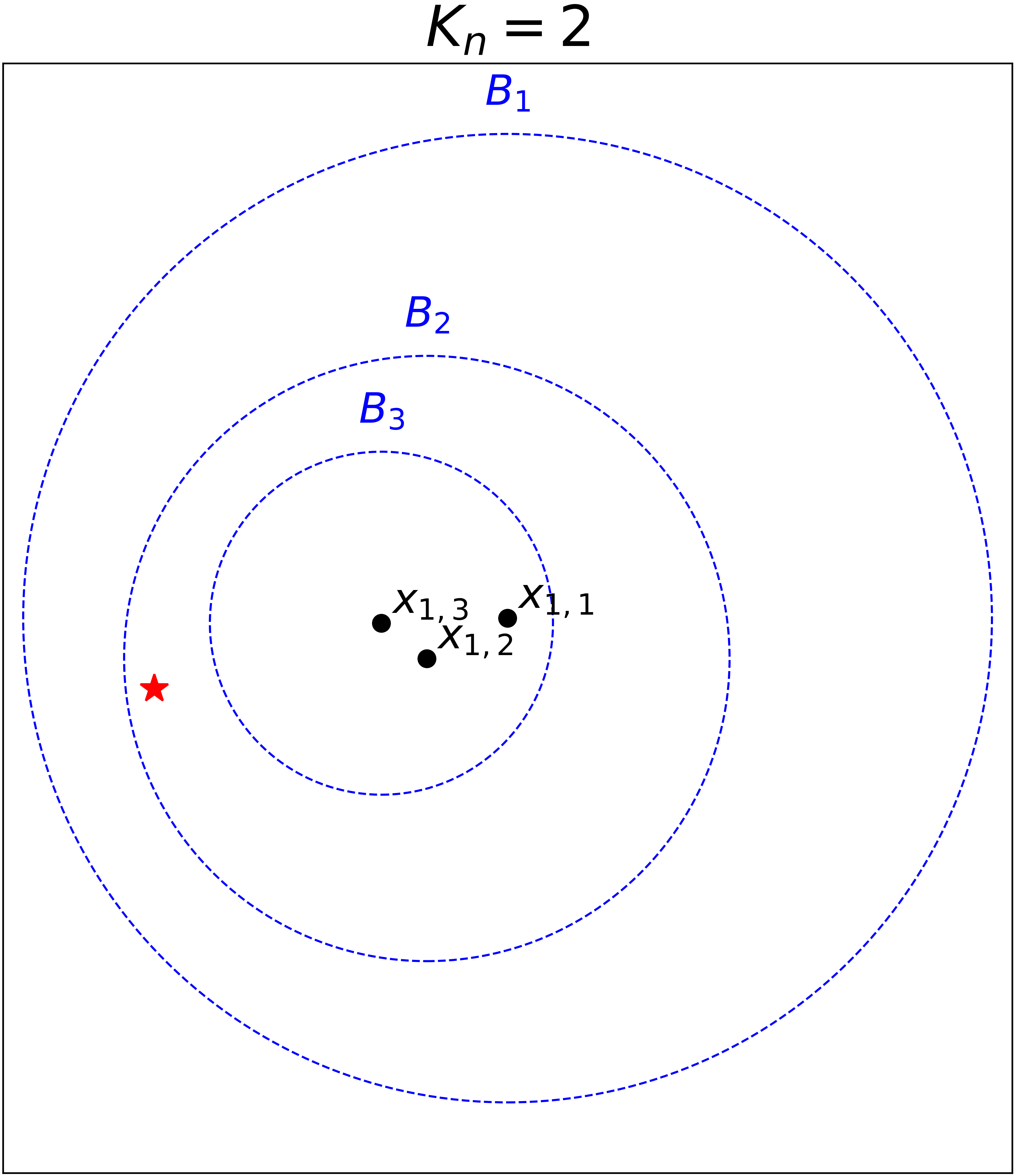}
    \end{minipage}%
    \begin{minipage}[b]{0.33\textwidth}
        \centering
        \includegraphics[width=\textwidth]{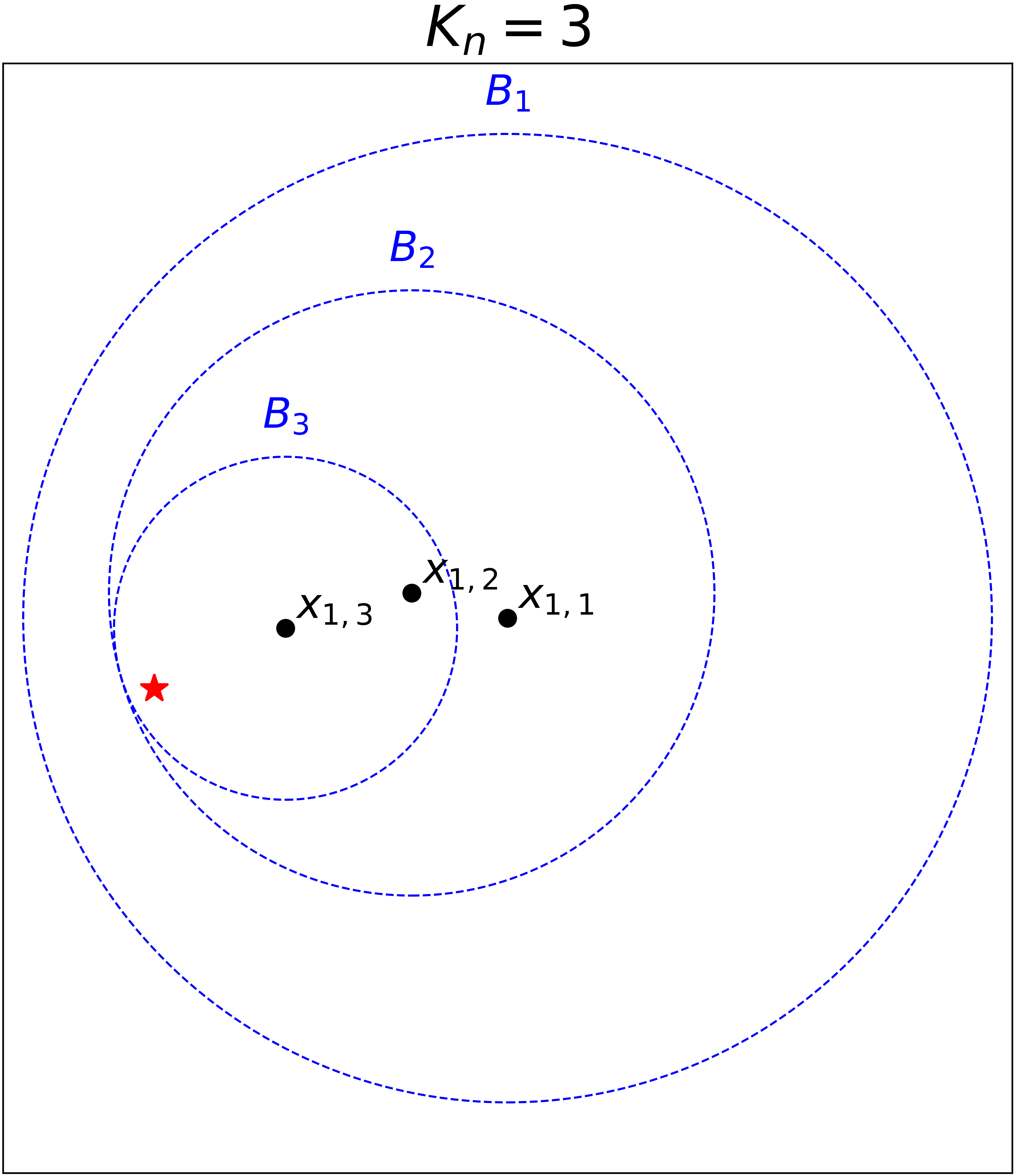}
    \end{minipage}

    \vspace{0.6em}

    % =======================
    % Bottom row: R_k schedule
    % =======================
    \includegraphics[width=\textwidth]{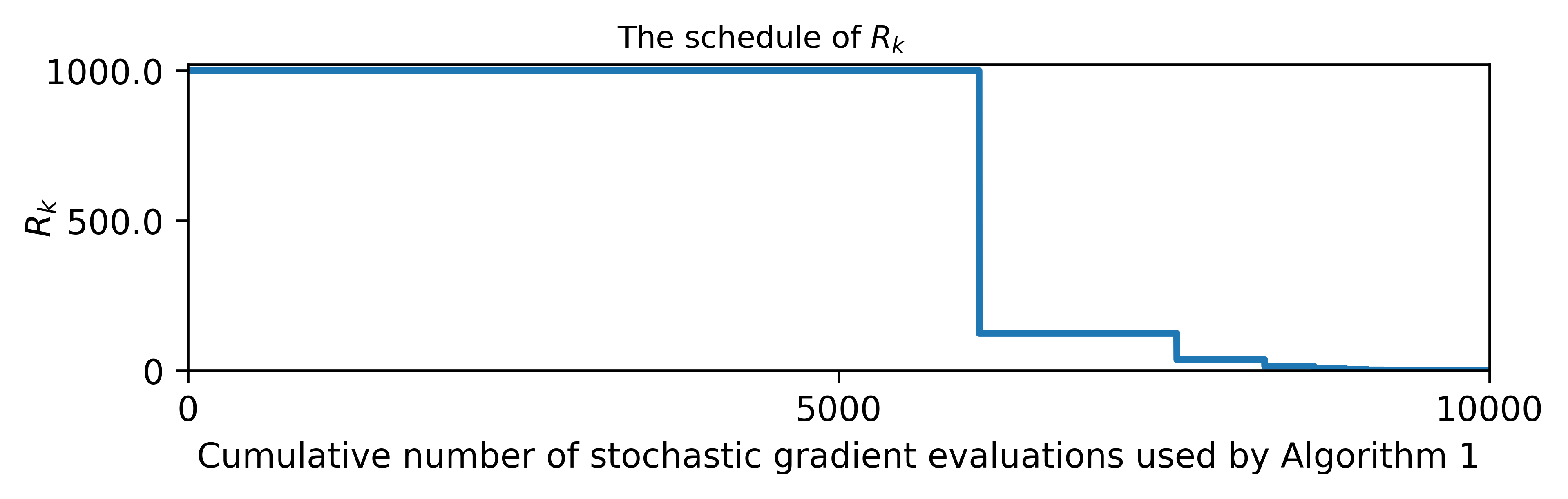}

    \vspace{-0.6em}
    \caption{
    Evolution of \Cref{alg:nested-adaptive} under increasing budget.
    \emph{Top:} optimization trajectories for $\kn=1,2,3$ (from left to right), where the red star denotes the optimal solution and $B_k$ denotes the ball $\ball{x_{1,k},R_k}$.
    \emph{Bottom:} the schedule of $R_k$ used in \Cref{alg:nested-adaptive} with $R_1=1000$ and total budget $n=20000$.
    Recall that the first $n/2=10,000$ stochastic gradient evaluations are used to run stage 1, i.e., compute $R_1$.
    }
    \label{fig:overall}
\end{figure}

The following \Cref{thm.stochastic case} provides a nonasymptotic bound on the probability that a shrinking ball contains a minimizer whenever $R_1$ is bounded below by a deterministic constant $r\in(0,4R]$.
The proof leverages well-established convergence results for \ADASGD{}:

\begin{fact}
    \label{fact.adasgd stochastic}
    Suppose \Cref{assume-nonsmooth-convex-optimization} holds. For any $k\le k_{\max}$ and $0<\delta_k<1$, fix an arbitrary realization of all randomness before outer loop $k$. Then, over the fresh samples in loop $k$, with probability at least $1-\delta_k$,
    \begin{equation}
        \label{eq.convergence adasgd stochastic high prob}
        F(x_{1,k+1})-\min _{x \in \ball{x_{1,k},R_k}} F(x) \leq \frac{10L R_k}{\sqrt{T_k}}\sqrt{\log (2/\delta_k)}.
    \end{equation}
    Also, for the same fixed realization, taking expectation over loop $k$ gives
    \begin{equation}
        \label{eq.convergence adasgd stochastic expectation}
        \mathbb{E}\!\left[F(x_{1,k+1})-\min _{x \in \ball{x_{1,k},R_k}} F(x)\right] \leq \frac{20L R_k}{\sqrt{T_k}}.
    \end{equation}
\end{fact}
This fact can be established by applying the well-known results of \citet[Theorems~3.10 and~4.14]{orabona2025modern} to isotropic \ADASGD~\citep[Section~3.3]{gupta2017unified}.

The high-level idea is that the progress achieved in each outer loop 
iteration can be carried over to the subsequent one and thus we can apply \Cref{lem:monotone-error-bound}.

\begin{lemma}
    \label{thm.stochastic case}
    Suppose \Cref{assume-nonsmooth-convex-optimization} holds. For any $1\le \bar k \le k_{\max}$, $\delta \in (0,1)$, and $r\in(0,4R]$, we have 
    \begin{equation}\label{eq.kb non asym}
    \pr\left(\mathcal{B}_{\budget,\bar k}^c\cap\mathcal{E}_{\lfloor \budget/2\rfloor}\cap\left\llbracket R_1\geq r\right\rrbracket\right)
    \leq\delta\,\pr\left(\mathcal{E}_{\lfloor \budget/2\rfloor}\cap\left\llbracket R_1\geq r\right\rrbracket\right)
    \quad\text{when}\quad \budget \geq \left(\frac{40\sqrt{2}\pi L R H_\delta}{\sqrt{3}\phi_{8R}(r/{\bar k}^3)}\right)^2.
    \end{equation}
    where $H_\delta := \sum_{k=1}^{\infty} \frac{\sqrt{\log (2^{k+1}/\delta)}}{k^2} <\infty$, and $\phi_{8R}$ denotes the strictly increasing function obtained by applying \Cref{lem:monotone-error-bound} to $F$ on $\ball{8R}$.
\end{lemma}
\begin{proof}
On $\mathcal{E}_{\lfloor\budget/2\rfloor}\cap\left\llbracket R_1\geq r\right\rrbracket$, we have $\dist{\mathbf{0}}{X^\star}\leq R_1$, so the conclusion holds for $\bar{k}=1$ deterministically. Now we consider the case $\bar{k}\geq2$.
For any $\delta\in(0,1)$, define $\delta_k := 2^{-k}\delta$, and the event
\begin{equation*}
\mathcal{G}_\budget
:= \bigcap_{k=1}^{k_{\max}}
\left\llbracket\eqref{eq.convergence adasgd stochastic high prob} \text{ holds for } k\right\rrbracket
\cap \mathcal{E}_{\lfloor \budget/2 \rfloor}
\cap\left\llbracket R_1\geq r\right\rrbracket.
\end{equation*}
Since $\mathcal{E}_{\lfloor \budget/2 \rfloor}\cap\left\llbracket R_1\geq r\right\rrbracket$ is determined after Stage I, \Cref{fact.adasgd stochastic} and the law of total probability show that the probability that this event occurs and the bound for loop $k$ fails is at most $\delta_k$ times the probability of this event. A union bound gives
\begin{equation}
\label{eq.property_delta_prob}
\pr\!\left(\mathcal{G}_\budget^c\cap
\mathcal{E}_{\lfloor \budget/2 \rfloor}\cap\left\llbracket R_1\geq r\right\rrbracket\right)
\leq \sum_{k=1}^{k_{\max}}\delta_k\,
\pr\!\left(\mathcal{E}_{\lfloor \budget/2 \rfloor}\cap\left\llbracket R_1\geq r\right\rrbracket\right)
\leq\delta\,\pr\!\left(\mathcal{E}_{\lfloor \budget/2 \rfloor}\cap\left\llbracket R_1\geq r\right\rrbracket\right).
\end{equation}

On $\mathcal{E}_{\lfloor \budget/2 \rfloor}$, we have $\dist{\mathbf{0}}{X^\star}\leq R_1\leq4R$. Moreover, for $k\leq k_{\max}$, $3\budget/(\pi^2k^2)\geq1$, and hence
\[
T_k=\left\lfloor\frac{3\budget}{\pi^2k^2}\right\rfloor\geq\frac{3\budget}{2\pi^2k^2}.
\]
Therefore, on $\mathcal{G}_\budget$,
\begin{align*}
F(x_{1,\bar{k}}) - F^\star &\le_{(a)} \sum_{k=1}^{\bar{k}-1} \left(F(x_{1,k+1})-\min _{x \in\ball{x_{1,k},R_k}} F(x)\right) \le_{(b)} \frac{10\sqrt{2}\pi L R_1}{\sqrt{3\budget}}\sum_{k=1}^{\bar{k}-1}\frac{\sqrt{\log (2/\delta_k)}}{k^2}\\
&\leq \frac{40\sqrt{2}\pi L R}{\sqrt{3\budget}}H_\delta,
\end{align*}
where $(a)$ uses that $\mathcal{B}_{\budget,1}$ occurs on $\mathcal{E}_{\lfloor \budget/2\rfloor}$ and $(b)$ uses \eqref{eq.convergence adasgd stochastic high prob}, $R_k=R_1/k^3$, and the preceding lower bound on $T_k$. Also,
\begin{equation*}
H_{\delta} = \sum_{k=1}^{\infty} \frac{\sqrt{\log (2/\delta_k)}}{k^2} \leq  \sum_{k=1}^{\infty} \frac{\sqrt{(k+1)\log 2+\log (1/\delta)}}{k^2}
\leq\sum_{k=1}^{\infty} \frac{\sqrt{k+1}}{k^2}\sqrt{\log 2} + \sum_{k=1}^{\infty} \frac{\sqrt{\log (1/\delta)}}{k^2}< \infty.
\end{equation*}

The update formula of \Cref{alg:nested-adaptive} and the property of Apéry's constant~\citep{apery1979irrationalite} give $\|x_{1,k}\| \leq \sum_{j=1}^{\infty} R_j \leq 2R_1\leq8R$. Define $\hat{F}(x):=F(x)$ on $\ball{8R}$ and let $H^\star:=X^\star\cap\ball{8R}$. This set is nonempty because $\dist{\mathbf{0}}{X^\star}\leq R$. By \Cref{lem:monotone-error-bound} and the inclusion $H^\star\subseteq X^\star$,
\[
\phi_{8R}(\dist{x_{1,\bar{k}}}{X^\star})
\leq\phi_{8R}(\dist{x_{1,\bar{k}}}{H^\star})
\leq F(x_{1,\bar{k}})-F^\star.
\]
Under the lower bound on $\budget$ in the lemma statement, the preceding inequalities yield
\[
\phi_{8R}(\dist{x_{1,\bar{k}}}{X^\star})
\leq \frac{40\sqrt{2}\pi L R}{\sqrt{3\budget}}H_\delta
\leq\phi_{8R}(r/\bar{k}^3).
\]
Since $\phi_{8R}$ is strictly increasing and $r\leq R_1$ on $\mathcal{G}_\budget$, we obtain
\[
\dist{x_{1,\bar{k}}}{X^\star}\leq\frac{r}{\bar{k}^3}
\leq\frac{R_1}{\bar{k}^3}=R_{\bar{k}}.
\]
Thus $\mathcal{B}_{\budget,\bar{k}}$ occurs. Combining this implication with \eqref{eq.property_delta_prob} proves the claim.
\end{proof}

\Cref{thm.stochastic case expectation} shows that our method has an $o(n^{-1/2})$ asymptotic rate by separately controlling the contributions from $\mathcal{E}_{\lfloor \budget/2 \rfloor}$ and $\mathcal{E}_{\lfloor \budget/2 \rfloor}^c$.

\begin{theorem}\label{thm.stochastic case expectation}
Suppose \Cref{assume-nonsmooth-convex-optimization} holds. The final output $\xout{\budget}$ of \Cref{alg:nested-adaptive} satisfies
\[        
\lim_{\budget\to \infty} \budget^{1/2}\mathbb{E}\left[F(\xout{\budget}) - F^\star \right] =0,
\]
where $\budget$ is the total number of stochastic gradient evaluations that \Cref{alg:nested-adaptive} uses.
\end{theorem}
\begin{proof}
First, we show that the total number of stochastic gradient evaluations is at most $\budget$.
Note that the stochastic gradient is evaluated once in each inner iteration of \Cref{alg:nested-adaptive}. Therefore, the total number of oracle calls in the second stage is upper-bounded by $\sum_{k=1}^{k_{\max}} T_k \leq \sum_{k=1}^{k_{\max}} \left\lfloor \tfrac{\budget}{k^2}\cdot\tfrac{3}{\pi^2}\right\rfloor\leq \sum_{k=1}^{\infty} \tfrac{\budget}{k^2}\cdot\tfrac{3}{\pi^2} = \budget/2$.
The last line is from the property of Basel series~(\cite[Chapter 9]{aigner1999proofs}).
Additionally, we use $\lfloor \budget/2 \rfloor$ evaluations for stage 1.

To prove our result, we show two separate results which imply the desired bound:
\begin{flalign}
\lim_{\budget\to \infty} \budget^{1/2}\mathbb{E}\left[\bigl(F(\xout{\budget}) - F^\star\bigr)\mathbf{1}_{\mathcal{E}_{\lfloor \budget/2 \rfloor}^c}\right] =0
\label{eq.bad event expectation} \\
\lim_{\budget\to \infty} \budget^{1/2}\mathbb{E}\left[\bigl(F(\xout{\budget}) - F^\star\bigr)\mathbf{1}_{\mathcal{E}_{\lfloor \budget/2 \rfloor}}\right] =0.
\label{eq.good event expectation}
\end{flalign}
We begin by establishing \eqref{eq.bad event expectation}.
For all sufficiently large $\budget$, we have $\Fbudget^{-\alpha}\leq L$. For such $\budget$, observe
\begin{flalign*}
F(\xout{\budget}) - F^\star &= F(\xout{\budget}) - F(\mathbf{0}) + F(\mathbf{0}) - F^\star \le L \| \xout{\budget} \| + F(\mathbf{0}) - F^\star \\
&\le 4 L \| \xB \| + F(\mathbf{0}) - F^\star \le   8 L^2 \Fbudget^{1/4} + F(\mathbf{0}) - F^\star,
\end{flalign*}
where we used $\|\xout{\budget}\|\leq\sum_{k=1}^{k_{\max}}R_k\leq2R_1=4\|\xB\|$ and \Cref{lem:iterate-bound}.
Thus,
\[
\budget^{1/2 }\mathbb{E}\left[\bigl(F(\xout{\budget}) - F^\star\bigr)\mathbf{1}_{\mathcal{E}_{\lfloor \budget/2 \rfloor}^c}\right] \le \budget^{1/2} (8 L^2 \Fbudget^{1/4} + F(\mathbf{0}) - F^\star) \pr(\mathcal{E}_{\lfloor \budget/2 \rfloor}^c).
\]
By \Cref{lem.xb converge non asym}, $\pr(\mathcal{E}_{\lfloor \budget/2 \rfloor}^c)$ decays stretched-exponentially in $\budget$. Since the remaining factor in the preceding bound grows only polynomially, their product converges to zero, which establishes \eqref{eq.bad event expectation}.
It remains to establish \eqref{eq.good event expectation}.
Fix any $\varepsilon\in(0,1)$ and integer $\bar{k}\geq2$, and take $\budget$ sufficiently large that $\bar{k}\leq k_{\max}$.
On $\mathcal{E}_{\lfloor\budget/2\rfloor}$, the event $\mathcal{B}_{\budget,1}$ occurs. Thus, on $\mathcal{E}_{\lfloor\budget/2\rfloor}\cap\left\llbracket R_1<\varepsilon R\right\rrbracket$, telescoping from outer loop $1$ gives
\[
F(\xout{\budget})-F^\star
\leq\sum_{k=1}^{k_{\max}}
\left(F(x_{1,k+1})-\min_{x\in\ball{x_{1,k},R_k}}F(x)\right).
\]
The event $\mathcal{E}_{\lfloor\budget/2\rfloor}\cap\left\llbracket R_1<\varepsilon R\right\rrbracket$ is determined after Stage I. Thus, \eqref{eq.convergence adasgd stochastic expectation}, the law of total expectation, $R_k=R_1/k^3$, and the lower bound on $T_k$ give
\begin{equation*}
\sqrt{\budget}\,\mathbb{E}\!\left[
\bigl(F(\xout{\budget})-F^\star\bigr)
\mathbf{1}_{\mathcal{E}_{\lfloor\budget/2\rfloor}\cap\left\llbracket R_1<\varepsilon R\right\rrbracket}
\right]
\leq\frac{20\sqrt{2}\pi L\varepsilon R}{\sqrt{3}}
\sum_{k=1}^\infty\frac{1}{k^2}.
\end{equation*}

On $\mathcal{E}_{\lfloor\budget/2\rfloor}\cap\left\llbracket R_1\geq\varepsilon R\right\rrbracket\cap\mathcal{B}_{\budget,\bar{k}}$, telescoping from outer loop $\bar{k}$ gives
\[
F(\xout{\budget})-F^\star
\leq\sum_{k=\bar{k}}^{k_{\max}}
\left(F(x_{1,k+1})-\min_{x\in\ball{x_{1,k},R_k}}F(x)\right).
\]
The events $\left\llbracket R_1\geq\varepsilon R\right\rrbracket$ and $\mathcal{B}_{\budget,\bar{k}}$ are determined before loop $\bar{k}$ begins. Thus, \eqref{eq.convergence adasgd stochastic expectation}, the law of total expectation, $R_1\leq4R$ on $\mathcal{E}_{\lfloor\budget/2\rfloor}$, and the lower bound on $T_k$ give
\begin{equation*}
\sqrt{\budget}\,\mathbb{E}\!\left[
\bigl(F(\xout{\budget})-F^\star\bigr)
\mathbf{1}_{\mathcal{E}_{\lfloor\budget/2\rfloor}\cap\left\llbracket R_1\geq\varepsilon R\right\rrbracket\cap\mathcal{B}_{\budget,\bar{k}}}
\right]
\leq\frac{80\sqrt{2}\pi L R}{\sqrt{3}}
\sum_{k=\bar{k}}^\infty\frac{1}{k^2}.
\end{equation*}
On the other hand, on $\mathcal{E}_{\lfloor\budget/2\rfloor}$, we have $\|\xout{\budget}\|\leq2R_1\leq8R$ and hence
\[
F(\xout{\budget})-F^\star\leq L\bigl(\|\xout{\budget}\|+\dist{\mathbf{0}}{X^\star}\bigr)\leq9LR.
\]
By \Cref{thm.stochastic case}, with $r=\varepsilon R$, the probability of the remaining failure event decays exponentially in $\budget$ for every fixed $\varepsilon$ and $\bar{k}$. Since $\sqrt{\budget}$ grows only polynomially and $F(\xout{\budget})-F^\star\leq9LR$ on $\mathcal{E}_{\lfloor\budget/2\rfloor}$, it follows that
\begin{equation*}
\sqrt{\budget}\,\mathbb{E}\!\left[
\bigl(F(\xout{\budget})-F^\star\bigr)
\mathbf{1}_{\mathcal{E}_{\lfloor\budget/2\rfloor}\cap\left\llbracket R_1\geq\varepsilon R\right\rrbracket\cap\mathcal{B}_{\budget,\bar{k}}^c}
\right]
\longrightarrow0.
\end{equation*}
The following three events form a partition of $\mathcal{E}_{\lfloor\budget/2\rfloor}$. Hence, the law of total expectation gives
{
\begin{align*}
&\mathbb{E}\!\left[\bigl(F(\xout{\budget})-F^\star\bigr)
\mathbf{1}_{\mathcal{E}_{\lfloor\budget/2\rfloor}}\right]\\
&\quad ={}\mathbb{E}\!\left[\bigl(F(\xout{\budget})-F^\star\bigr)
\mathbf{1}_{\mathcal{E}_{\lfloor\budget/2\rfloor}\cap\left\llbracket R_1<\varepsilon R\right\rrbracket}\right]\\[-1pt]
&\qquad +\mathbb{E}\!\left[\bigl(F(\xout{\budget})-F^\star\bigr)
\mathbf{1}_{\mathcal{E}_{\lfloor\budget/2\rfloor}\cap\left\llbracket R_1\geq\varepsilon R\right\rrbracket\cap\mathcal{B}_{\budget,\bar{k}}}\right]\\[-1pt]
&\qquad +\mathbb{E}\!\left[\bigl(F(\xout{\budget})-F^\star\bigr)
\mathbf{1}_{\mathcal{E}_{\lfloor\budget/2\rfloor}\cap\left\llbracket R_1\geq\varepsilon R\right\rrbracket\cap\mathcal{B}_{\budget,\bar{k}}^c}\right].
\end{align*}
}
Consequently,
\[
\limsup_{\budget\to\infty}
\sqrt{\budget}\,\mathbb{E}\!\left[
\bigl(F(\xout{\budget})-F^\star\bigr)
\mathbf{1}_{\mathcal{E}_{\lfloor\budget/2\rfloor}}
\right]
\leq\frac{20\sqrt{2}\pi L\varepsilon R}{\sqrt{3}}
\sum_{k=1}^\infty\frac{1}{k^2}
+\frac{80\sqrt{2}\pi L R}{\sqrt{3}}
\sum_{k=\bar{k}}^\infty\frac{1}{k^2}.
\]
Letting $\bar{k}\to\infty$ and then $\varepsilon\downarrow0$ establishes \eqref{eq.good event expectation}.
\end{proof}

\begin{remark}
\Cref{alg:nested-adaptive} requires a budget, $\budget$, for the number of stochastic gradient evaluations as input; however, using the doubling trick of \citet[Section~2.3.1]{shalev2025online}, one can readily convert it into an anytime algorithm.
\end{remark}

\section{A tight asymptotic lower bound for stochastic convex optimization}\label{sec.lower bound}
For standard nonasymptotic bounds for stochastic convex optimization, we construct a hard instance of the form
\[
    f(x;1) := \abs{x-1}, \quad f(x;-1) := \abs{x+1}, \quad \text{where} \quad \pr(S = \pm 1) = \frac{1 \pm \epsilon v}{2}.
\]
Any algorithm that can find an $\epsilon$-approximately optimal solution to this problem, with success probability greater than $1/5$, also implicitly provides an estimator of $v$. Setting $\epsilon \approx 1/\sqrt{\budget}$ and choosing $v$ to be a Rademacher random variable $V$, i.e., (taking the value of $1$ or $-1$ with equal chance) gives a problem which one can show requires at least $\budget$ samples $S\ind{1}, \dots, S\ind{\budget}$ to estimate $V$ with success probability greater than $1/5$ \citep{carmon2024price,agarwal2009information}. Unfortunately, this approach fundamentally depends on a priori knowledge of $\budget$ to decide the problem instances. Once the algorithm exceeds this budget, better convergence rates, even exponentially fast convergence rates \citep{davis2024stochastic}, can be achieved on this instance.

To provide an asymptotic lower bound we need a construction that does not require $\budget$ as input.
Roughly speaking, our approach involves constructing an infinite sequence of hard problems with larger and larger $\budget$, all nested into a single one-dimensional problem instance. The nesting is determined by a sequence of nested intervals, defined as follows.
Let $a_1 := -1, b_1:=-1/2, c_1:=1/2, d_1 := 1$ and for $i > 1$ let
\begin{equation}\label{define:A_i-B_i-C_i-D_i}
\begin{aligned}
a_i &:= \begin{cases} 
a_{i-1} & \text{if } v_{i-1} = -1 \\
c_{i-1} & \text{if } v_{i-1} = 1,
\end{cases} \quad\quad\quad\quad & d_i &:= \begin{cases} 
b_{i-1} & \text{if } v_{i-1} = -1 \\
d_{i-1} & \text{if } v_{i-1} = 1,
\end{cases}  \\
b_i &:= \frac{3}{4} a_i + \frac{1}{4} d_i, & c_i &:= \frac{1}{4} a_i + \frac{3}{4} d_i.
\end{aligned}
\end{equation}
Later, in the proof of our main result, when $v_i$ is replaced with a random variable $V_i$ we will replace $a_i$, $b_i$, $c_i$ and $d_i$ with $A_i$, $B_i$, $C_i$, $D_i$ to signify that they become random variables.

Our sample functions comprise of the following piecewise linear convex functions,
\begin{flalign}\label{eq:define-f_i}
f_i(x;a_i,b_i,c_i,d_i,s_i) := \begin{cases} 
0 & s_i = 0 \\
-x+a_i &s_i = -1 ~~\&~~ x<a_i \\
\frac{1-\epsilon_i}{1+\epsilon_i}(x-a_i) & s_i = -1 ~~\&~~ a_i\leq x< b_i \\
x-\frac{2\epsilon_i}{1+\epsilon_i}b_i-\frac{1-\epsilon_i}{1+\epsilon_i}a_i & s_i = -1 ~~\&~~ b_i \le x \\
-x+\frac{2\epsilon_i}{1+\epsilon_i}c_i+\frac{1-\epsilon_i}{1+\epsilon_i}d_i & s_i=1 ~~~~~\&~~~  x<c_i \\
-\frac{1-\epsilon_i}{1+\epsilon_i}(x-d_i) & s_i = 1 ~~~~~\&~~~  c_i\leq x<d_i \\
x-d_i & s_i = 1 ~~~~~\&~~~ d_i \le x.
\end{cases}
\end{flalign}
To construct our instance, we parameterize a family of three-valued distributions:
\begin{equation}
%\label{eq.parameterized three-valued}
S \sim \threval(w,\epsilon, v),
\end{equation}
where $v \in \{-1, 1\}$, $w \in (0,1)$, $\epsilon \in (0,1)$, with 
\begin{equation*}
\pr(S=-1 ) = w \times \frac{1-\epsilon v}{2}, \quad \pr(S=0) = 1-w, \quad \pr(S=1) = w \times \frac{1+\epsilon v}{2}.
\end{equation*}
With abuse of notation let 
\[
f_i(x;Q_{1:i-1},S_i) := f_i(x;a_i,b_i,c_i,d_i,S_i)
\]
where for all $i \in \NN$,
\begin{subequations}\label{eq:S-i-Q-i-combined-distribution}
\begin{flalign}
S_i &\sim \threval\left(w_i,\epsilon_i,v_i\right) \label{eq:S_i-distribution} \\
Q_i &:= \begin{cases}
v_i & \text{ if } \exists j > i \text{ s.t. } S_{j} \neq 0 \\
0 & \text{otherwise}.
\end{cases}
\end{flalign}
\end{subequations}
Further, define the expectation of the $i$-th sample function as
\begin{equation}\label{eq:ith-combined-function}
F_i(x; v) := \mathbb{E}[f_i(x;Q_{1:i-1},S_i)] 
= w_i\left[\frac{1+v_i\epsilon_i}{2}f_i(x;a_i,b_i,c_i,d_i,1) + \frac{1-v_i\epsilon_i}{2}f_i(x;a_i,b_i,c_i,d_i,-1)\right].
\end{equation}
We visualize $f_i$ and $F_i$ in \Cref{fig:main_figure}.

\begin{figure}[!ht]
    \centering
    \begin{minipage}[b]{0.48\textwidth}
        \centering
        \includegraphics[width=\textwidth]{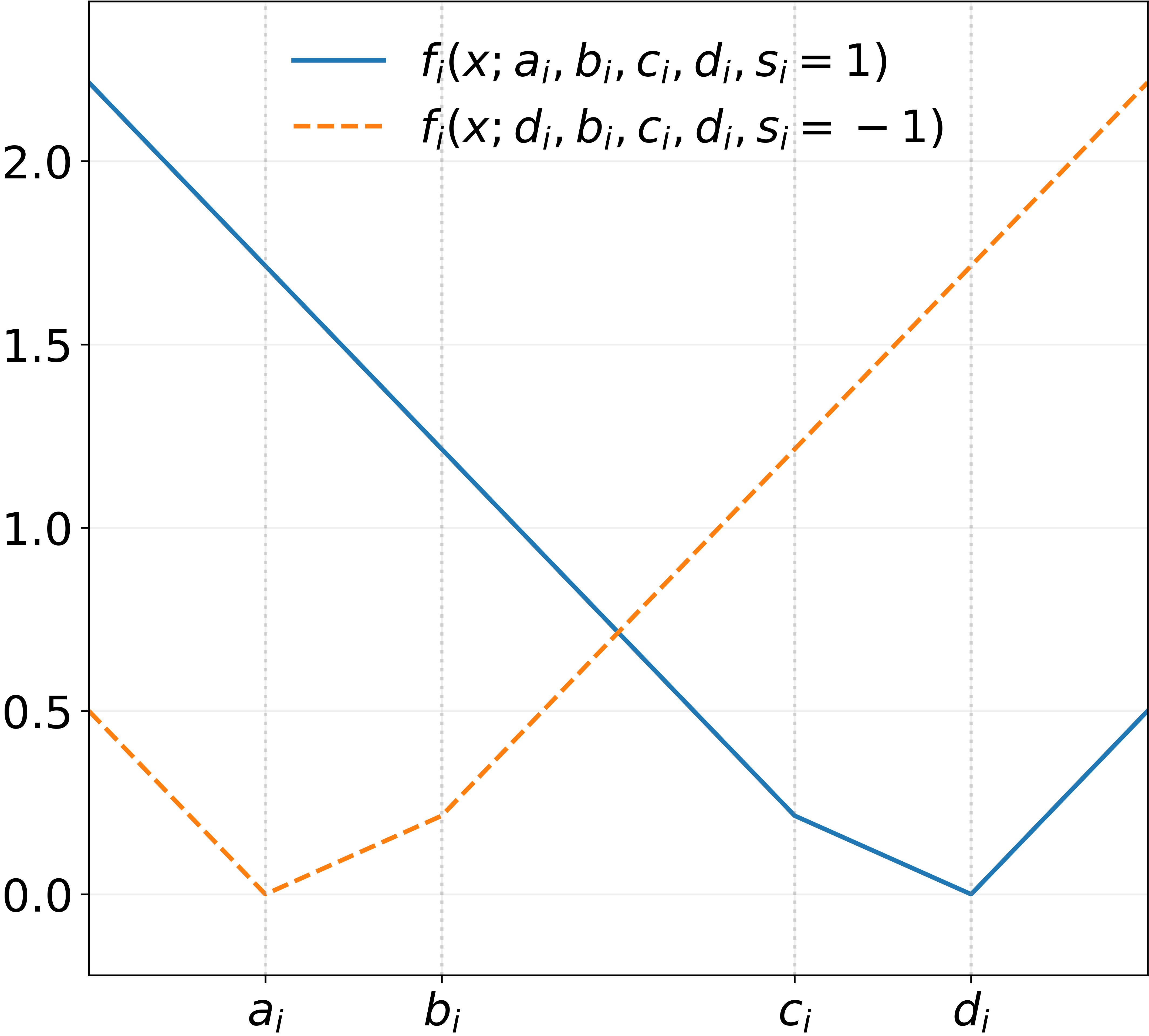}
       % \label{fig:sample}
    \end{minipage}
    % \hfill
    \begin{minipage}[b]{0.48\textwidth}
        \centering
        \includegraphics[width=\textwidth]{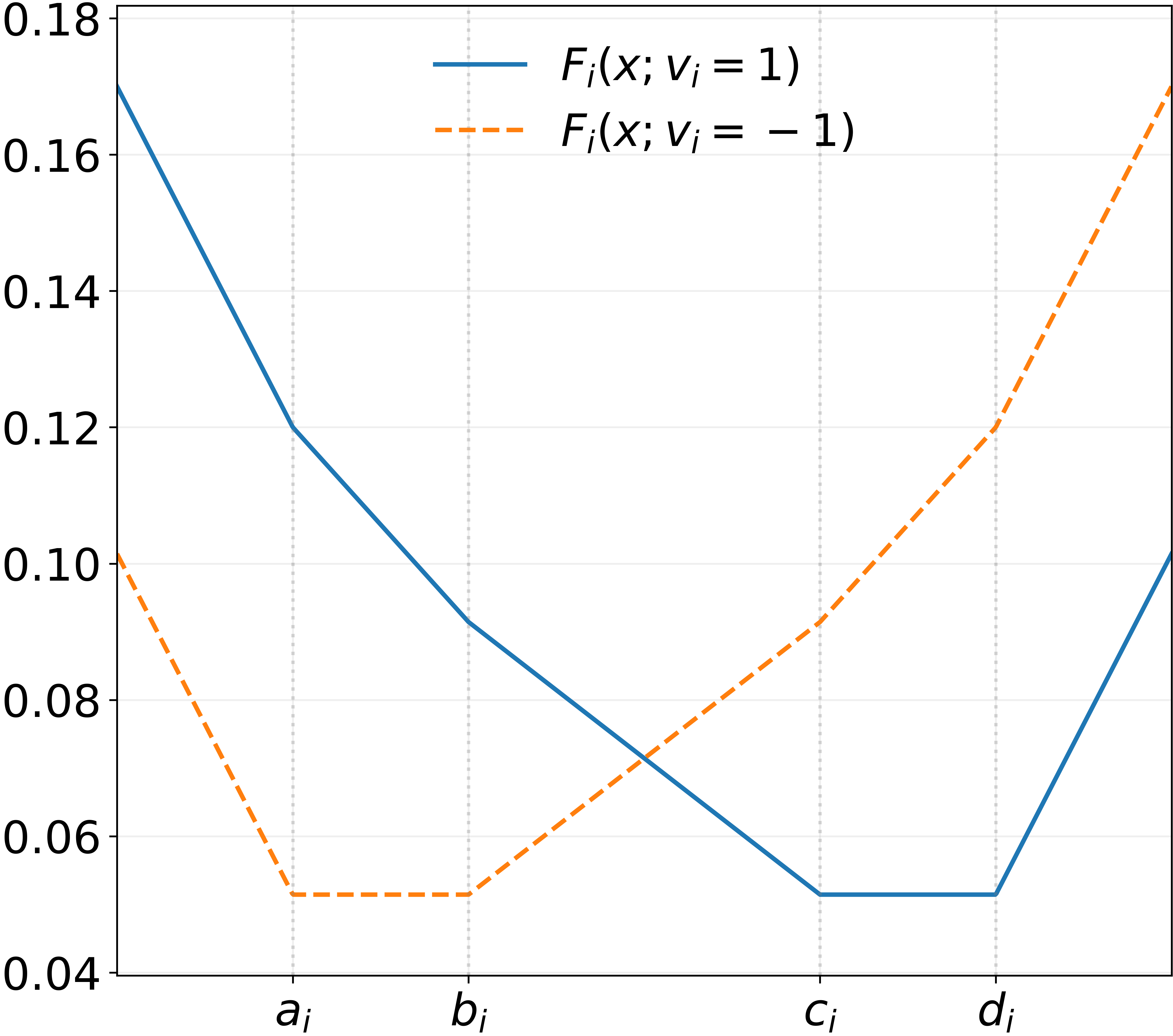}
       % \label{fig:combined}
    \end{minipage}
    \caption{Plot of $i$th sample function $f_i(x,a_i,b_i,c_i,d_i,s_i)$ for $s_i=-1$ and $s_i=1$ (left), and its expectation, $F_i(x,v_i)$ for $v_i=-1$ and $v_i=1$ (right), with $\epsilon=0.4$, $w=0.1$. Note that $F_i(x;  1)$ is flat on the interval $[a_i,b_i]$ and $F_i(x; -1)$ is flat on the interval $[c_i,d_i]$.}
    \label{fig:main_figure}
\end{figure}
\noindent Denote the optimal value and the interval of optimal solutions to the expectation of the $i$th sample function as
\[
F_i^\star(v) := \inf_{x\in\mathbb{R}} F_i(x;v)
\quad\text{and}\quad
X_i^\star(v) := \argmin_{x\in\mathbb{R}} F_{i}(x; v).
\]
Now, we introduce the formal construction of the hard instance:
\[
f(x; Q, S) := \sum_{i=1}^{\infty} 2^{-i} f_i(x; Q_{1:i-1}, S_i) 
\]
and thus 
\begin{flalign}\label{eq:F-x-V-formula}
F_{f,P_v}(x)=\E_{Q,S \sim P_v}\left[ \sum_{i=1}^{\infty}2^{-i} f_i(x; Q_{1:i-1}, S_i) \right] = \sum_{i=1}^{\infty}2^{-i} F_{i}(x; v)
\end{flalign}
where $P_v$ denotes our distribution over $Q$ and $S$, i.e., \eqref{eq:S-i-Q-i-combined-distribution}. We visualize $F_{f,P_v}$ in \Cref{fig:zoom}.

\begin{figure}[ht]
    \centering
    \includegraphics[width=\textwidth]{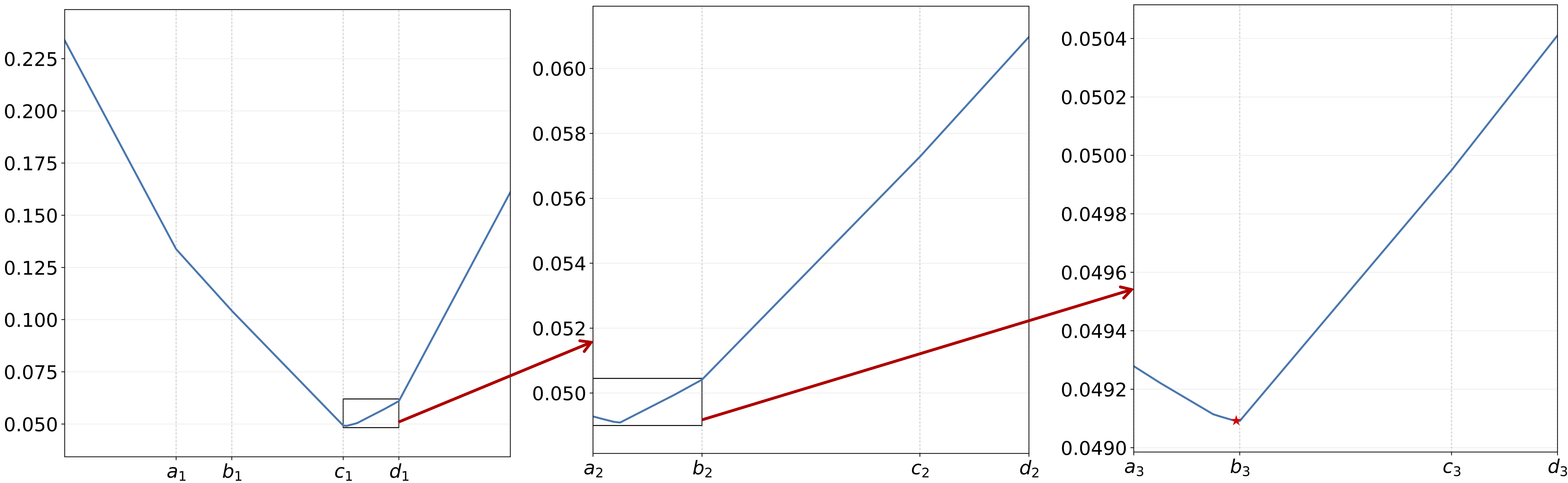}
    \vspace{-1.5em}
    \caption{Visualization of the hard instance
    $F_{f,P_{v}}(x)$, truncated at $i=10$
    ($v_2=v_3=0$, $v_i=1$ for $i\neq 2,3$):
    full view (left), zoom-in on $[c_1,d_1]$ (middle), zoom-in on $[a_2,b_2]$ (right).
    }
    \label{fig:zoom}
\end{figure}

We examine the basic properties of our construction in~\Cref{lem:basic-properties-of-construction}, which shows the suboptimality gap for each $F_i(x;v)$ and establishes that the hard instance is valid in the sense that it satisfies~\Cref{assume-nonsmooth-convex-optimization}.
\begin{lemma}[Basic properties of our construction]\label{lem:basic-properties-of-construction}
Let $x \in \R$, $v \in \{-1, 1 \}^{\NN}$ and $w_i, \epsilon_i \in (0,1)$. Then, for all $i \in \NN$,
\begin{equation}
\label{eq.suboptimality Fi}
F_{i}(x; v)-F_i^\star(v) \geq w_i\epsilon_i \dist{x}{X_i^\star(v)},
\end{equation}
and
\begin{flalign}\label{eq:ith-optimal-set-form}
X_i^\star(v) = \begin{cases}
[a_i, b_i] & v_i = -1 \\
[c_i, d_i] & v_i = 1. 
\end{cases}
\end{flalign}
Additionally, the problem instance $(f,P_v)$ satisfies \Cref{assume-nonsmooth-convex-optimization} with $R=L=1$.
\end{lemma}

\begin{proof}
    For the proof of \eqref{eq.suboptimality Fi} and \eqref{eq:ith-optimal-set-form}, we only consider the case where $v_i = -1$, since the case where $v_i=1$ follows by symmetry. 
    From \eqref{eq:define-f_i}, \eqref{eq:S-i-Q-i-combined-distribution}, and \eqref{eq:ith-combined-function}, we have
\[
F_i(x;-1) = w_i
\begin{cases}
    \begin{aligned}
        -x &+ \frac{1-\epsilon_i}{2} \left( \frac{2\epsilon_i}{1+\epsilon_i}c_i + \frac{1-\epsilon_i}{1+\epsilon_i}d_i \right) \\
        &+ \frac{1+\epsilon_i}{2}a_i,
    \end{aligned} 
    & x < a_i, \\[1.5em]
    \frac{1-\epsilon_i}{2} \left( \frac{2\epsilon_i}{1+\epsilon_i}c_i + \frac{1-\epsilon_i}{1+\epsilon_i}d_i - a_i \right), 
    & a_i \le x < b_i, \\[1.2em]
    \begin{aligned}
        \epsilon_i x &+ \frac{1-\epsilon_i}{2} \left( \frac{2\epsilon_i}{1+\epsilon_i}c_i + \frac{1-\epsilon_i}{1+\epsilon_i}d_i \right) \\
        &- \frac{1+\epsilon_i}{2} \left( \frac{2\epsilon_i}{1+\epsilon_i}b_i + \frac{1-\epsilon_i}{1+\epsilon_i}a_i \right),
    \end{aligned} 
    & b_i \le x < c_i, \\[2.2em]
    \begin{aligned}
        \frac{2\epsilon_i}{1+\epsilon_i}x &- \frac{1+\epsilon_i}{2} \left( \frac{2\epsilon_i}{1+\epsilon_i}b_i + \frac{1-\epsilon_i}{1+\epsilon_i}a_i \right) \\
        &+ \frac{(1-\epsilon_i)^2}{2(1+\epsilon_i)}d_i,
    \end{aligned} 
    & c_i \le x < d_i, \\[2.2em]
    \begin{aligned}
        x &- \frac{1+\epsilon_i}{2} \left( \frac{2\epsilon_i}{1+\epsilon_i}b_i + \frac{1-\epsilon_i}{1+\epsilon_i}a_i \right) \\
        &- \frac{1-\epsilon_i}{2}d_i,
    \end{aligned} 
    & d_i \le x.
\end{cases}
\]
    This shows that $F_i(x;-1)$ is piecewise linear and convex; hence its subdifferential can be characterized explicitly at every point.
    In particular, for any $x\in [a_i,b_i]$, the subdifferential contains $0$, which immediately yields \eqref{eq:ith-optimal-set-form}. Moreover, \eqref{eq.suboptimality Fi} follows by applying the subgradient inequality
    \[
    F_i(x;-1) \geq F_i(y;-1) + g_i(y;-1)\cdot(x-y), \quad g_i(y;-1) \in \partial F_i(y;-1),
    \]
    with $y=\Pi_{X_i^\star(v)}(x)$ and $g_i(y;-1) = \arg\max_{g\in \partial F_i(y;-1)} |g|$ using that $g_i(a_i; -1) = -w_i$ and $g_i(b_i; -1) = w_i \epsilon_i$.

    By \Cref{define:A_i-B_i-C_i-D_i} and \eqref{eq:ith-optimal-set-form}, $X_{i+1}^\star(v) \subseteq X_{i}^\star(v)$ and \[\text{len}(X_{i+1}^\star(v)) = \frac{1}{4} \text{len}(X_{i}^\star(v)).\] Thus, $\lim_{i\to \infty} \text{len}(X_i^\star(v))=0$ because the initial interval is bounded.

    Then we can apply the Nested Interval Theorem~\cite[Theorem 2.38]{rudin1976principles}, i.e., there exists a unique point $x^\star \in \bigcap_{i=1}^{\infty} X_i^\star(v)$. Also from the \eqref{eq:F-x-V-formula}, we conclude that $x^\star$ is the unique minimizer of $F_{f,P_v}$~($\|x^\star\| \leq 1$ since $X_1^\star(v) = [-1,1]$). 
    
    In addition, for any $q,s$ in the sample space defined in \eqref{eq:S-i-Q-i-combined-distribution} and any $x,y\in \R$, we have
    \begin{equation*}
        \begin{aligned}
        \left| f(x;q,s)-f(y;q,s)\right| &= \left|\sum_{i=1}^\infty 2^{-i} \left(f_i(x;q_{1:i-1},s_i)-f_i(y;q_{1:i-1},s_i)\right)\right|\\
        &\le \sum_{i=1}^\infty 2^{-i} \left|f_i(x;q_{1:i-1},s_i)-f_i(y;q_{1:i-1},s_i)\right| \le_{(a)} \sum_{i=1}^\infty 2^{-i}|x-y| \\
        &= |x-y| \sum_{i=1}^\infty 2^{-i} = |x-y|
        \end{aligned}
    \end{equation*}
    $(a)$ is from the definition in \eqref{eq:define-f_i}. Therefore we verified \Cref{assume-nonsmooth-convex-optimization} holds.
\end{proof}

% Now we are ready to show how to pick $\epsilon_i$ and $w_i$ such that any stochastic algorithm will perform poorly over the hard instance we construct.

Before presenting our lower bound, we formally define a stochastic optimization method.
Our definition is very broad and encompasses methods that have unrestricted access to the sample functions $f(\cdot ; \cdot)$ but 
observe $P$ only through the samples $(S\ind{1}, Q\ind{1}), \dots, (S\ind{\budget}, Q\ind{\budget}) \simIID P$. More precisely, the output of a stochastic optimization method after observing $\budget$ samples is
\begin{flalign}\label{eq:measurable-mapping}
\xout{\budget} = \text{A}_{\budget}\big( f, U, (S\ind{1}, Q\ind{1}), \dots, (S\ind{\budget}, Q\ind{\budget}) \big)
\end{flalign}
where $\text{A}_\budget$ is some measurable mapping to $\R^{d}$ and $U \sim \mathsf{Unif}([0,1])$ is independent of both $S\ind{1}, \dots, S\ind{\budget}$ and $Q\ind{1}, \dots , Q\ind{\budget}$, enabling additional randomization.

\Cref{thm.lower bound} shows that for any arbitrarily slowly decreasing bijection from $[1, \infty) \rightarrow (0,1]$, we can generate a distribution over problem instances $(f,P_V)$ where $V$ is an infinite sequence of random variables and that is asymptotically hard for any stochastic optimization method to optimize. If the method knew $V$, this would be straightforward, since it could reconstruct $F_{f,P_V}$ given $V$ and is allowed arbitrary computational resources to optimize it. However, as implied by \eqref{eq:measurable-mapping}, the method can only estimate $V$ through $S\ind{1}, \dots, S\ind{\budget}$ and $Q\ind{1}, \dots, Q\ind{\budget}$.

The proof of \Cref{thm.lower bound}
generates a distribution over instances $(f,P_V)$ by letting each $V_i$ be a Rademacher random variable (i.e., equal to $-1$ or $1$ with equal probability).
We imagine that we have an efficient method for minimizing $F_{f,P_V}$ and then use that method to generate an estimator for $V_i$ based on the idea that (approximate) knowledge $X_i^\star(V)$ implies knowledge of $V_i$.
To complete the proof we use a lemma that we prove in \Cref{sec:hardness-of-estimation} on the hardness of estimating the random variable $V_i$ given the information provided from observing $S\ind{1}_i, \dots, S_i\ind{\budget}$ and $Q\ind{1}_i, \dots, Q_i\ind{\budget}$. We can focus on these random variables because the remaining random variables do not provide any additional useful information for estimating $V_i$.

\begin{theorem}[Asymptotic lower bound on constant probability guarantees]
\label{thm.lower bound}
Let $\littleFunc : [1,\infty) \rightarrow (0,1]$ be a decreasing bijection.
Then,
there exists a distribution over 
problem instances, $(f,P_V)$, with each realization satisfying \Cref{assume-nonsmooth-convex-optimization} for $L=R=1$, and a sequence $\budget_1, \budget_2, \dots$ with $\lim_{k \rightarrow \infty} \budget_k = \infty$, such that for any stochastic optimization method that samples $\budget_i$ functions from $(f,P_V)$ and outputs $\xout{\budget_i}$,
\[
\pr\bigg( F_{f,P_{V}}(\xout{\budget_i})- F_{f,P_{V}}^\star \ge \frac{\littleFunc(\budget_i)}{\sqrt{\budget_i}} \bigg) \ge 1/6 \quad \forall i \in \NN.
\]
\end{theorem}

\begin{proof}
%{\color{red}careful about $i=1$ case!}
Let $\budget_0 = 1$ and consider the following sequences generated for $i \ge 1$,
\begin{flalign}\label{eq:sequence-definitions}
w_i = \frac{1}{12 \budget_{i-1}}, \quad \budget_i = \max\{ 2 \budget_{i-1}, \lceil \littleFunc^{-1}(w_i 8^{-i}) \rceil \}, \quad \text{ and } \quad \epsilon_i = \frac{1}{\sqrt{\budget_i}}.
\end{flalign}   
The first part of this proof shows that
\begin{flalign}\label{large-dist-implies-gap}
\dist{\xout{\budget_i}}{X_i^\star(v)} \ge 4^{-i} \implies F_{f,P_v}(\xout{\budget_i})- F_{f,P_v}^\star \ge \frac{\littleFunc(\budget_i)}{\sqrt{\budget_i}}.
\end{flalign}
By \Cref{lem:basic-properties-of-construction} and the premise that $\dist{\xout{\budget_i}}{X_i^\star(v)} \ge 4^{-i}$ we have
\begin{equation*}
F_{f,P_v}(\xout{\budget_i})- F_{f,P_v}^\star \ge 2^{-i} (F_{i}(\xout{\budget_i}; v)- F_i^\star(v))\geq 2^{-i} w_i \epsilon_i \dist{\xout{\budget_i}}{X_i^\star(v)} \ge w_i\epsilon_i 8^{-i}.
\end{equation*}
Combining this with the expression for $\budget_i$ and $\epsilon_i$ given in \eqref{eq:sequence-definitions}, and $\budget_i \ge \littleFunc^{-1}(w_i 8^{-i}) \implies \littleFunc(\budget_i) \le w_i 8^{-i}$ yields
\begin{equation}
\label{eq.suboptimality non asym}
F_{f,P_v}(\xout{\budget_i})- F_{f,P_v}^\star \geq w_i\epsilon_i 8^{-i} = \frac{w_i 8^{-i}}{\sqrt{\budget_i}} \ge \frac{\littleFunc(\budget_i)}{\sqrt{\budget_i}}.
\end{equation}
Suppose that each $V_i$ is a Rademacher random variable and define the estimator
\[
\hat{V}_i := \begin{cases}
~~~1 & \text{if } \xout{\budget_i} > (A_i + D_i)/2 \\
-1 &  \text{if } \xout{\budget_i} \le (A_i + D_i)/2.
\end{cases}
\]
Then,
\begin{flalign}\label{eq:dist-big-prob}
\pr( \dist{\xout{\budget_i}}{X_i^\star(V)} \ge 4^{-i}) \ge_{(a)} \pr(V_i \neq \hat{V}_i) \ge_{(b)} \frac{1}{3} - \sum_{t=1}^{\budget_i} \pr(Q_i\ind{t} \neq 0) \ge_{(c)} \frac{1}{6} 
\end{flalign}
where $(a)$ holds because $V_i \neq \hat{V}_i$ implies that 
$\dist{\xout{\budget_i}}{X_i^\star(V)} \ge \frac{d_i - a_i}{4} = (d_1 - a_1) 4^{1-i} / 4  \geq 4^{-i}$,  $(b)$ uses \Cref{lem:V-information-theory} and the fact that the randomized estimator $\hat{V}_i$ depends on $V_i$ only through $S_i\ind{1:\budget_i}$ and $Q_{i}\ind{1:\budget_i}$, i.e., $V_i \independent \hat{V}_i \mid S_i\ind{1:\budget_i}, Q_{i}\ind{1:\budget_i}$, and $(c)$ uses that 
\begin{equation*}
\pr(Q\ind{t}_i \neq 0) = \pr(\exists j > i \text{~~s.t.~~} S\ind{t}_j \neq 0) \le \sum_{j=i+1}^{\infty}\pr(S\ind{t}_j \neq 0) \le \sum_{j=i+1}^{\infty} w_j \le \frac{1}{12 \budget_i} \sum_{i=0}^{\infty} 2^{-i}\le \frac{1}{6\budget_i}.
\end{equation*}
Therefore, by \Cref{large-dist-implies-gap} and \eqref{eq:dist-big-prob},
\[
\pr \left( F_{f,P_{V}}(\xout{\budget_i})- F_{f,P_{V}}(x^\star)\geq \frac{\littleFunc(\budget_i)}{\sqrt{\budget_i}} \right) \ge \pr( \dist{\xout{\budget_i}}{X_i^\star(V)} \ge 4^{-i}) \ge \frac{1}{6}.
\]
\end{proof}

\Cref{coro:asymptotic-lower-bound} provides a lower bound on the expected suboptimality that is slightly cleaner but weaker than \Cref{thm.lower bound}. \Cref{coro:asymptotic-lower-bound} follows immediately from \Cref{thm.lower bound} and Markov's inequality.
This result matches the upper bound we obtained in \Cref{thm.stochastic case expectation}.

\begin{corollary}[Asymptotic lower bound on expectation guarantees]
\label{coro:asymptotic-lower-bound}
Let $\littleFunc : [1,\infty) \rightarrow (0,1]$ be a decreasing bijection.
Then,
there exists a distribution over 
problem instances, $(f,P_V)$, with each realization satisfying \Cref{assume-nonsmooth-convex-optimization} for $L=R=1$, such that, for any stochastic optimization method that samples $\budget$ functions from $(f,P_V)$ and outputs $\xout{\budget}$,
\[
\limsup_{\budget \rightarrow \infty} \frac{\budget^{1/2}  \E[F_{f,P_{V}}(\xout{\budget})- F_{f,P_{V}}^\star]}{\littleFunc(\budget)} \ge 1/6.
\]
\end{corollary}

\begin{remark}
It is worth contrasting \Cref{coro:asymptotic-lower-bound} with standard nonasymptotic lower bounds. For example,  
\citet{agarwal2009information} show that there exists a problem-independent constant $C > 0$ such that for every $\budget$ there exists a distribution over problem instances $(f,P_{V}^{\budget})$, with each realization satisfying \Cref{assume-nonsmooth-convex-optimization} for $L=R=1$, such that $n^{1/2} \E[F_{f,P_{V}^{\budget}}(\xout{\budget})- F_{f,P_{V}^{\budget}}^\star] \ge C > 0$. Critically, their distribution over `hard instances' is specific to each value of $n$, whereas our lower bound has no such dependence.
\end{remark}

\section{Discussion: asymptotic vs. nonasymptotic convergence guarantees}\label{sec:discussion}

While this is a paper about asymptotic optimization, it is important to highlight that there are several strengths of nonasymptotic guarantees over asymptotic guarantees.
Firstly, nonasymptotic guarantees offer a more granular measure of the performance of a method as they include the impact of other problem parameters such as the Lipschitz constant and distance. 
Indeed, it is possible to evaluate nonasymptotic guarantees on instances, for example, if we have crude estimates of the Lipschitz constants or distances to optimality. Conversely, if we see poor convergence of a method with such guarantees, we can examine the possibility that the quantities in our convergence bound, e.g., the Lipschitz constant, are very high for our particular instance. Second, in real stochastic optimization problems, particularly those arising in machine learning, it is rare that we fix the function that we are optimizing and indefinitely increase the amount of data. 
For example, in the compute limited regime (i.e., training LLMs) we train a sequence of compute optimal models where the number of steps (i.e., the data) and the model are simultaneously changing. In such regimes standard asymptotic guarantees are meaningless but it is possible that we can estimate the Lipschitz constants and distance to optimality for these sequences of problems. 

\section*{Statement on AI usage}
AI was used in this paper to help find errors in proofs, edit grammar, and
perform literature search. It was also helpful for the authors to 
understand \citet{diakonikolas2020lower} and adapt 
their result to our setting. More specifically, initial drafts of the proof of \Cref{lem:direct-sum-regularity} and of the modified proofs of \citet[Lemmas~4--6]{woodworth2018graph} (\Cref{lem:hard-instance-unbounded-queries} in this paper) were prepared with AI assistance. The authors carefully checked and rewrote these materials and take full responsibility for all results, claims, and the final presentation.
\section*{Acknowledgments}
We thank John Duchi, Francesco Orabona and Yair Carmon for helpful discussions.
Yuntian Jiang gratefully acknowledges the support of his advisor Bo Jiang
from Shanghai University of Finance and Economics.
\appendix
\section{Supplement to \texorpdfstring{\Cref{sec:general-lower-bounds}}{Section \ref*{sec:general-lower-bounds}}}\label{sec:app-general-lb}
\subsection{Well-definedness and regularity of the direct sums}

\begin{lemma}
\label{lem:direct-sum-regularity}
Let $\{d_i\}_{i\geq1}$ be a sequence of positive integers and
$x=(\xcom{i})_{i\geq1}\in\hilbert{\infty}$.
\begin{enumerate}[i.]
    \item If $\funcom{i}\in\funclsncvx[d_i]$ for every $i$, then
    \(
    F(x):=\sum_{i=1}^{\infty}2^{-i}\funcom{i}(\xcom{i})
    \)
    is well defined, belongs to $\lipsfunc{\hilbert{\infty}}{p}$,
    and satisfies
    $\|\grad^rF(\mathbf{0})\|\leq1$ for every $r\in[p]$.

    \item If $\funcom{i}\in\funclscvx[d_i]$ for every $i$, then
    \(
    F(x):=\sum_{i=1}^{\infty}
    2^{-\frac{p+3}{2}i}\funcom{i}(2^{i/2}\xcom{i})
    \)
    is well defined, belongs to $\lipsfunc{\hilbert{\infty}}{p}$,
    and satisfies
    $\|\grad^rF(\mathbf{0})\|\leq1$ for every $r\in[p]$.
\end{enumerate}
\end{lemma}

\begin{proof}
We use the convention $\grad^0 f=f$. For every component function in
either part, Taylor's theorem and the conditions in \Cref{def:fun-cls} give,
for $0\leq r\leq p$,
\begin{equation}
\label{eq:direct-sum-taylor-bound}
\|\grad^r\funcom{i}(z)\|
\leq
\sum_{j=r}^{p}\frac{\|z\|^{j-r}}{(j-r)!}
+\frac{\|z\|^{p+1-r}}{(p+1-r)!}.
\end{equation}

For part i, let
$F_N(x):=\sum_{i=1}^N2^{-i}\funcom{i}(\xcom{i})$.
If $M>N$ and $\|x\|\leq R$, then, for $0\leq r\leq p$,
\[
\|\grad^rF_M(x)-\grad^rF_N(x)\|
\leq
\sum_{i=N+1}^{M}2^{-i}\|\grad^r\funcom{i}(\xcom{i})\|\leq
\left(
\sum_{j=r}^{p}\frac{R^{j-r}}{(j-r)!}
+\frac{R^{p+1-r}}{(p+1-r)!}
\right)
\sum_{i=N+1}^{M}2^{-i}.
\]
Thus, the partial sums and their derivatives through order $p$ converge
uniformly on every bounded subset of $\hilbert{\infty}$. Applying Dieudonn\'e's convergence theorem for Fr\'echet derivatives
\citep[Theorem~(8.6.3)]{dieudonne1960foundations}
successively to the partial sums and their derivatives shows that
$F$ is $p$ times continuously differentiable and that its derivatives are the
corresponding term-by-term sums. In particular, for $r\in[p]$,
\[
\|\grad^rF(\mathbf{0})\|
\leq
\sum_{i=1}^{\infty}2^{-i}
\|\grad^r\funcom{i}(\mathbf{0})\|
\leq\sum_{i=1}^{\infty}2^{-i}=1.
\]
For any $x,y\in\hilbert{\infty}$, the $p$th-order Lipschitz continuity of
the component functions and Cauchy--Schwarz give
\[
\|\grad^pF(x)-\grad^pF(y)\|
\leq
\sum_{i=1}^{\infty}2^{-i}\|\xcom{i}-\ycom{i}\|\leq
\left(\sum_{i=1}^{\infty}4^{-i}\right)^{1/2}\|x-y\|
=\frac{1}{\sqrt3}\|x-y\|.
\]
This proves part i.

For part ii, set
$\lambda_i:=2^{-(p+3)i/2}$, $\mu_i:=2^{i/2}$, and
$F_N(x):=\sum_{i=1}^N\lambda_i\funcom{i}(\mu_i\xcom{i})$.
By \eqref{eq:direct-sum-taylor-bound}, for $0\leq r\leq p$,
$M>N$, and $\|x\|\leq R$,
\[
\|\grad^rF_M(x)-\grad^rF_N(x)\|
\leq
\sum_{j=r}^{p}\frac{R^{j-r}}{(j-r)!}
\sum_{i=N+1}^{M}2^{-\frac{p+3-j}{2}i}\quad+
\frac{R^{p+1-r}}{(p+1-r)!}
\sum_{i=N+1}^{M}2^{-i}.
\]
Every series on the right is geometric. Hence, the partial sums and their
derivatives through order $p$ converge uniformly on bounded subsets, so the
same convergence theorem for Fr\'echet derivatives shows that $F$ is $p$
times continuously differentiable and that its derivatives are the term-by-term sums. For $r\in[p]$,
\[
\|\grad^rF(\mathbf{0})\|
\leq
\sum_{i=1}^{\infty}
2^{-\frac{p+3-r}{2}i}\|\grad^r\funcom{i}(\mathbf{0})\|\leq
\sum_{i=1}^{\infty}2^{-3i/2}<1.
\]
Finally, since $\lambda_i\mu_i^{p+1}=2^{-i}$,
\(
\|\grad^pF(x)-\grad^pF(y)\|
\leq
\sum_{i=1}^{\infty}2^{-i}\|\xcom{i}-\ycom{i}\|\leq\frac{1}{\sqrt3}\|x-y\|.
\)
This proves part ii.
\end{proof}

\subsection{Proof of \texorpdfstring{\Cref{coro:convex-fixed-rand}}{Corollary \ref*{coro:convex-fixed-rand}}}
\label{app:convex-fixed-rand}
For every $d\in\NN$, let
$\mathcal C_d^{\mathrm{cvx,bd}}$, $\mathcal F_d^{\rm ns}$, and
$\mathcal F_d^{\rm sm}$ denote the $d$-dimensional versions of
$\mathcal C_\infty^{\mathrm{cvx,bd}}$, $\mathcal F_\infty^{\rm ns}$, and
$\mathcal F_\infty^{\rm sm}$, respectively, as defined in
\Cref{coro:convex-fixed-rand}. We use the error metric
$\epsilon(x,f):=f(x)-f^\star$.
We prove the corollary by verifying
\Cref{assm:error-metric} and \Cref{assm:general nonasymptotic rand},
and then applying \Cref{thm.lower-bound-fixed-function-convex-rand}.

\begin{lemma}
\label{lem:convex-constrained-error-metric}
Each of the two dimension-indexed families
$\{\mathcal{F}^{\rm ns}_d\}_{d\in\NN\cup\{\infty\}}$ and
$\{\mathcal{F}^{\rm sm}_d\}_{d\in\NN\cup\{\infty\}}$, together with the error metric
\(
    \epsilon(x,f)=f(x)-f^\star,
\)
satisfies \Cref{assm:error-metric}.
\end{lemma}
\begin{proof}
Take \(q=0\), \(\lambda_i=2^{-2i}\), and \(\mu_i=2^{i/2}\).  Consider any
sequence of dimensions \(\{\dimFunc_i\}_{i\ge1}\) and functions
\(\funcom{i}\in\mathcal{F}^{\rm ns}_{\dimFunc_i}\) for every $i$
(respectively, \(\funcom{i}\in\mathcal{F}^{\rm sm}_{\dimFunc_i}\) for every $i$), and let
\(
    F(x):=\sum_{i=1}^\infty \lambda_i\funcom{i}(\mu_i\xcom{i}), \qquad x=(\xcom{i})_{i\ge1}\in\hilbert\infty.
\)

Convexity follows from the blockwise convexity of the summands. Moreover, in the case of $\funcls_{d}^{\rm ns}$
the objective series is well defined, since for every
$x\in\hilbert\infty$,
\[
\sum_{i=1}^\infty\lambda_i|\funcom{i}(\mu_i\xcom{i})|
\leq
\sum_{i=1}^\infty\lambda_i
+\left(\sum_{i=1}^\infty\lambda_i^2\mu_i^2\right)^{1/2}\|x\|
<\infty.
\]
For any \(x,y\in\hilbert\infty\),
\[
    |F(x)-F(y)|
    \le
    \sum_{i=1}^\infty \lambda_i\mu_i\|\xcom{i}-\ycom{i}\|
    \le
    \left(\sum_{i=1}^\infty 2^{-3i}\right)^{1/2}\|x-y\|
    \le \|x-y\|.
\]
Therefore \(F\) is \(1\)-Lipschitz.
Moreover, any selected subgradient of each $\funcom{i}$ has norm at
most $1$. Hence the blockwise selected subgradient belongs to
$\hilbert\infty$ because its squared norm is at most
$\sum_{i=1}^\infty\lambda_i^2\mu_i^2=\sum_{i=1}^\infty2^{-3i}<\infty$;
summing the component subgradient inequalities shows that it belongs to
$\partial F(x)$.

In the case of $\funcls_{d}^{\rm sm}$,
choose $x_i^\star\in\argmin\funcom{i}$ such that
$\|x_i^\star\|=\dist{\mathbf 0}{\argmin\funcom{i}}\leq1$. Since
$\grad\funcom{i}(x_i^\star)=0$, $1$-smoothness gives
$\|\grad\funcom{i}(\mathbf 0)\|\leq\|x_i^\star\|\leq1$. Hence, for every
$x\in\hilbert\infty$,
\[
\begin{aligned}
\sum_{i=1}^\infty\lambda_i|\funcom{i}(\mu_i\xcom{i})|
&\leq \sum_{i=1}^\infty\lambda_i|\funcom{i}(\mathbf 0)|
+\sum_{i=1}^\infty\lambda_i\mu_i\|\xcom{i}\|
+\frac12\sum_{i=1}^\infty\lambda_i\mu_i^2\|\xcom{i}\|^2\\
&\leq \sum_{i=1}^\infty\lambda_i
+\left(\sum_{i=1}^\infty\lambda_i^2\mu_i^2\right)^{1/2}\|x\|
+\frac12\left(\sup_i\lambda_i\mu_i^2\right)\|x\|^2<\infty,
\end{aligned}
\]
so the objective series is well defined. Define
$G(x):=(\lambda_i\mu_i\grad\funcom{i}(\mu_i\xcom{i}))_{i\geq1}$. Then
\[
\|G(\mathbf 0)\|^2
\leq\sum_{i=1}^\infty\lambda_i^2\mu_i^2
=\sum_{i=1}^\infty2^{-3i}<\infty,
\]
and $1$-smoothness gives
\[
\|G(x)-G(\mathbf 0)\|^2
\leq\sum_{i=1}^\infty\lambda_i^2\mu_i^4\|\xcom{i}\|^2
\leq\|x\|^2.
\]
Thus $G(x)\in\hilbert\infty$. Finally, for every
$h=(\underline h_i)_{i\geq1}\in\hilbert\infty$,
\[
|F(x+h)-F(x)-\langle G(x),h\rangle|
\leq\frac12\sum_{i=1}^\infty\lambda_i\mu_i^2\|\underline h_i\|^2
\leq\frac12\|h\|^2,
\]
which proves that $F$ is Fr\'echet differentiable with
\(
    \grad F(x)
    =
    G(x)
\).
Using the \(1\)-smoothness of each \(\funcom{i}\),
\[
\begin{aligned}
    \|\grad F(x)-\grad F(y)\|^2
    &=
    \sum_{i=1}^\infty
    \lambda_i^2\mu_i^2
    \|\grad\funcom{i}(\mu_i\xcom{i})
      -\grad\funcom{i}(\mu_i\ycom{i})\|^2  \\
    &\le
    \sum_{i=1}^\infty
    \lambda_i^2\mu_i^4\|\xcom{i}-\ycom{i}\|^2
    \le
    \|x-y\|^2,
\end{aligned}
\]
because \(\sup_i\lambda_i\mu_i^2=\sup_i2^{-i}\le1\).

We next verify the minimizer conditions. In both cases, for each \(i\),
choose \(x_i^\star\in\argmin \funcom{i}\) such that
\(\|x_i^\star\|=\dist{\mathbf 0}{\argmin\funcom{i}}\leq1\).
Set
\(
    x^\star
    :=
    (\mu_1^{-1}x_1^\star,\mu_2^{-1}x_2^\star,\ldots).
\)
Then
\(
    \|x^\star\|^2
    =
    \sum_{i=1}^\infty \mu_i^{-2}\|x_i^\star\|^2
    \le
    \sum_{i=1}^\infty 2^{-i}
    =
    1.
\)
Also, since each summand is minimized block-wise at \(x_i^\star\),
\(
    F^\star
    =
    F(x^\star)
    =
    \sum_{i=1}^\infty \lambda_i\funcom{i}^\star .
\)
Because \(-1\le\funcom{i}^\star\le \funcom{i}(\mathbf 0)\le1\) and
\(\sum_i\lambda_i=\sum_i2^{-2i}<1\), we get
\(
    -1\le F^\star\le F(\mathbf 0)\le1.
\)
Since \(x^\star\in\argmin F\) and \(\|x^\star\|\leq1\), we have
\(\dist{\mathbf 0}{\argmin F}\leq1\).

Finally, for every \(x=(\xcom{i})_{i\ge1}\in\hilbert\infty\),
\[
    \epsilon(x,F)
    =
    F(x)-F^\star =
    \sum_{j=1}^\infty
    \lambda_j
    \bigl(\funcom{j}(\mu_j\xcom{j})-\funcom{j}^\star\bigr) \ge
    \lambda_i
    \bigl(\funcom{i}(\mu_i\xcom{i})-\funcom{i}^\star\bigr)
    =
    \lambda_i\epsilon(\mu_i\xcom{i},\funcom{i})
\]
for every \(i\in\NN\). This is exactly \eqref{eq:error-metric} with
\(q=0\).
\end{proof}

We next verify \Cref{assm:general nonasymptotic rand}. The argument
follows the proof of \citet[Theorem~1]{woodworth2018graph}, specialized
to the setting of a sequential first-order oracle graph. There are two
minor modifications. First, following
the construction of \citet[Equation~(3)]{diakonikolas2020lower}, we introduce an additional
maximum operator, which adapts the construction to the unconstrained
optimization setting. Second, we restate the vector guessing event so
that it holds with probability at least \(1-2\delta\), rather than with
constant probability \(1/2\), by increasing the ambient dimension by a
logarithmic factor in \(1/\delta\).

Now, we introduce the hard instance used to establish the nonasymptotic
lower bound. The construction incorporates the technique of
\citet{diakonikolas2020lower}, which allows the lower bound analysis of
\citet{woodworth2018graph} to be extended to the unconstrained
optimization setting. Fix \(\budget\in\NN\). For a parameter \(H\in\{1,\infty\}\),
let
\[
    \ell
    =
    \begin{cases}
        \frac13, & H=\infty,\\[2mm]
        \frac{1}{10(\budget+1)^{3/2}}, & H=1,
    \end{cases}
    \qquad
    \eta
    =
    10(\budget+1)^{3/2}\ell .
\]
Thus, for \(H=\infty\), \(\eta=\frac{10}{3}(\budget+1)^{3/2}\), while for \(H=1\),
\(\eta=1\).

Let \(v_1,\ldots,v_{\budget+1}\) be an orthonormal set drawn uniformly at random from
the unit sphere in \(\hilbert{d}\). Define
\[
    \tilde f(x)
    =
    \max\left\{
        \max_{1\le r\le \budget+1}
        \left(
            \ell v_r^T x
            -
            \frac{5\ell^2(r-1)}{\eta}
        \right),
        \ \|x\|-1-\frac{\ell}{\sqrt{\budget+1}}
    \right\}.
\]
Let \(f\) be the \(\eta\)-Moreau envelope of \(\tilde f\):
\begin{equation}\label{eq:hard-func-woodworth}
    f(x)
    =
    \inf_y
    \left\{
        \tilde f(y)
        +
        \frac{\eta}{2}\|y-x\|^2
    \right\}.
\end{equation}
The random draw of \(V=(v_1,\ldots,v_{\budget+1})\) induces a distribution over
functions \(f\).  Next we introduce the properties of the hard function; the first three properties are almost identical to those of \citet[Lemmas~4--6]{woodworth2018graph}. The last property prevents the algorithm from cheating by making unbounded queries.
\begin{lemma}
\label{lem:hard-instance-unbounded-queries}
The following properties hold for \(f\) defined in
\eqref{eq:hard-func-woodworth}.
\begin{enumerate}[i.]
    \item If \(H=\infty\), then \(f\in\funcls_d^{\rm ns}\), and if
    \(H=1\), then \(f\in\funcls_d^{\rm sm}\).

    \item If \(x\) satisfies
    \(
        |v_{\budget+1}^T x|
        \le
        \frac{\ell}{\eta},
    \)
    then
    \(
        f(x)-f^\star
        \ge
        \frac{\ell}{2\sqrt{\budget+1}}.
    \)
    Consequently, when \(H=\infty\),
    \(
        f(x)-f^\star
        \ge
        \frac{1}{6\sqrt{\budget+1}}
        \ge
        \frac{1}{6\sqrt2}\budget^{-1/2},
    \)
    and when \(H=1\),
    \(
        f(x)-f^\star
        \ge
        \frac{1}{20(\budget+1)^2}
        \ge
        \frac{1}{80}\budget^{-2}.
    \)

    \item If a query point \(x\) satisfies
    \(
        |\langle x,v_r\rangle|
        \le
        \frac{\ell}{\eta}\) for all $r\ge t$,
    then both \(f(x)\) and \(\grad f(x)\) can be computed using only
    \(v_1,\ldots,v_t\).

    \item If \(\|x\|>3\), then both \(f(x)\) and \(\grad f(x)\) can be
    computed without using \(v_1,\ldots,v_{\budget+1}\). Moreover,
    \(
        f(x)-f^\star
        \ge
        \frac{\ell}{2\sqrt{\budget+1}}.
    \)

\end{enumerate}
\end{lemma}

\begin{proof}
Let
\[
    g(y)
    :=
    \max_{1\le r\le \budget+1}
    \left\{
        \ell v_r^T y
        -
        \frac{5\ell^2(r-1)}{\eta}
    \right\},
    h(y):=\|y\|-1-\frac{\ell}{\sqrt{\budget+1}}.
\]
Thus \(\tilde f(y)=\max\{g(y),h(y)\}\).

First, \(g\) is convex and \(\ell\)-Lipschitz, while \(h\) is convex and
\(1\)-Lipschitz. Hence \(\tilde f\) is convex and \(1\)-Lipschitz. By the
standard property of the Moreau envelope~\citep[Proposition 12.29]{bauschke2017convex},
\(f\) is convex, $1$-Lipschitz and \(\eta\)-smooth, and 
\[
    \grad f(x)=\eta(x-y_x), \qquad y_x
    =
    \argmin_y
    \left\{
        \tilde f(y)+\frac{\eta}{2}\|y-x\|^2
    \right\}
\]
When \(H=1\), we have \(\eta=1\), so \(f\) is
\(1\)-smooth.

We next verify the minimizer conditions. Let
\(
    \bar x
    :=
    -\sum_{r=1}^{\budget+1}\frac{v_r}{\sqrt{\budget+1}}.
\)
Then \(\|\bar x\|=1\), and
\(
    g(\bar x)
    =
    -\frac{\ell}{\sqrt{\budget+1}},
    h(\bar x)
    =
    1-1-\frac{\ell}{\sqrt{\budget+1}}
    =
    -\frac{\ell}{\sqrt{\budget+1}}.
\)
Therefore
\(
    \tilde f^\star
    \le
    \tilde f(\bar x)
    =
    -\frac{\ell}{\sqrt{\budget+1}}.
\)
On the other hand, for every \(y\),
\(
    g(y)\ge \ell v_1^T y \ge -\ell\|y\|,
    h(y)=\|y\|-1-\frac{\ell}{\sqrt{\budget+1}},
\)
and hence
\[
    \tilde f(y)
    \ge
    \max\{-\ell\|y\|,\|y\|-1-\frac{\ell}{\sqrt{\budget+1}}\}.
\]
The right-hand side is coercive and its minimum over \(\|y\|\ge0\) equals
\(
    -\frac{\ell (1+\frac{\ell}{\sqrt{\budget+1}})}{1+\ell}.
\)
Therefore, we have
\(
    \tilde f^\star
    \ge
    -\ell
    \ge
    -1.
\)
Also, \(\tilde f(0)=0\), and therefore \(f(0)\le0\le1\). Since the Moreau
envelope has the same minimum value as \(\tilde f\), we get
\(
    -1
    \le
    f^\star
    \le
    f(0)
    \le
    1.
\)

The coercivity established above ensures that
\(\argmin\tilde f\neq\emptyset\). We now verify
\(\dist{\mathbf 0}{\argmin f}\leq1\). If \(\|y\|>1\), then
\[
    h(y)
    =
    \|y\|-1-\frac{\ell}{\sqrt{\budget+1}}
    >
    1-1-\frac{\ell}{\sqrt{\budget+1}}
    =
    -\frac{\ell}{\sqrt{\budget+1}}.
\]
Since \(\tilde f^\star\le-\ell/\sqrt{\budget+1}\), such a point \(y\) cannot be a
minimizer of \(\tilde f\). Hence
\(
    \argmin \tilde f\subseteq \ball 1.
\)
The Moreau envelope has the same minimizers as \(\tilde f\). Therefore
\(
    \argmin f=\argmin\tilde f\subseteq\ball 1.
\)
In particular, \(\dist{\mathbf 0}{\argmin f}\leq1\).
This proves the first claim.

For the second claim, we already know
\(
    f^\star
    =
    \tilde f^\star
    \le
    -\frac{\ell}{\sqrt{\budget+1}}.
\)
Now fix \(x\) such that
\(
    |v_{\budget+1}^T x|
    \le
    \frac{\ell}{\eta},
\)
and let
\(
    y^\star
    =
    \argmin_y
    \left\{
        \tilde f(y)+\frac{\eta}{2}\|y-x\|^2
    \right\}.
\)
By optimality,
\(
    \eta(x-y^\star)\in\partial\tilde f(y^\star).
\)
Since \(\tilde f=\max\{g,h\}\), there exist \(\beta,\gamma\ge0\) with
\(\beta+\gamma=1\), coefficients \(\alpha_r\ge0\) with
\(\sum_{r=1}^{\budget+1}\alpha_r=1\), and a vector \(q\in\partial\|y^\star\|\), such that
\(
    \eta(x-y^\star)
    =
    \beta\ell\sum_{r=1}^{\budget+1}\alpha_r v_r+\gamma q.
\)
Here \(\alpha_r=0\) unless the \(r\)-th affine piece is active in the
maximum defining \(g(y^\star)\); if \(\beta=0\), the choice of
\((\alpha_r)_{r=1}^{\budget+1}\) is arbitrary. Also, \(\gamma=0\) if the norm branch
is not active. If \(y^\star=0\), then \(v_{\budget+1}^T y^\star=0\). Otherwise,
\(q=y^\star/\|y^\star\|\), and taking inner product with \(v_{\budget+1}\) gives
\[
    \left(\eta+\frac{\gamma}{\|y^\star\|}\right)v_{\budget+1}^T y^\star
    =
    \eta v_{\budget+1}^T x-\beta\ell\alpha_{\budget+1}.
\]
Using \(|v_{\budget+1}^T x|\le\ell/\eta\), \(0\le\beta\le1,\alpha_{\budget+1}\le1\), and
\(\eta+\gamma/\|y^\star\|\ge\eta\), we obtain
\(
    v_{\budget+1}^T y^\star
    \ge
    -\frac{2\ell}{\eta}.
\)
Therefore
\[
    f(x)
    =
    \tilde f(y^\star)+\frac{\eta}{2}\|y^\star-x\|^2 
    \ge
    g(y^\star) 
    \ge
    \ell v_{\budget+1}^T y^\star
    -
    \frac{5\ell^2 \budget}{\eta} 
    \ge
    -\frac{5\ell^2 (\budget+1)}{\eta}.
\]
Combining this with \(f^\star\le-\ell/\sqrt{\budget+1}\), we get
\(
    f(x)-f^\star
    \ge
    \frac{\ell}{\sqrt{\budget+1}}
    -
    \frac{5\ell^2 (\budget+1)}{\eta}
    =
    \frac{\ell}{2\sqrt{\budget+1}},
\)
where we used \(\eta=10(\budget+1)^{3/2}\ell\). This proves the second claim.

For the third claim, suppose that
\(
    |\langle x,v_r\rangle|
    \le
    \frac{\ell}{\eta}\) for all $r\ge t.$
Let \(y^\star\) be the proximal point of \(x\). Using the same optimality
condition as above, for every \(r\ge t\), either \(y^\star=0\), in which case
\(\langle y^\star,v_r\rangle=0\), or
\(
    \left(\eta+\frac{\gamma}{\|y^\star\|}\right)
    \langle y^\star,v_r\rangle
    =
    \eta\langle x,v_r\rangle-\beta\ell\alpha_r.
\)
As in the second claim, we have
\(
    -\frac{2\ell}{\eta}
    \le
    \langle y^\star,v_r\rangle
    \le
    \frac{\ell}{\eta},
    r\ge t.
\)
In particular,
\(
    |\langle y^\star,v_r\rangle|
    \le
    \frac{2\ell}{\eta}\) for all $r\ge t.$
Thus, for any \(r>t\),
\[
    \ell v_r^T y^\star
    -
    \frac{5\ell^2(r-1)}{\eta}
    \le
    \frac{2\ell^2}{\eta}
    -
    \frac{5\ell^2(r-1)}{\eta},
\]
whereas
\[
    \ell v_t^T y^\star
    -
    \frac{5\ell^2(t-1)}{\eta}
    \ge
    -\frac{2\ell^2}{\eta}
    -
    \frac{5\ell^2(t-1)}{\eta}.
\]
For \(r>t\), the first upper bound is strictly smaller than the second
lower bound. Hence no index \(r>t\) can be active in the maximum defining
\(g(y^\star)\).

Define
\[
    g_t(y)
    :=
    \max_{1\le r\le t}
    \left\{
        \ell v_r^T y
        -
        \frac{5\ell^2(r-1)}{\eta}
    \right\},
    \qquad
    \tilde f_t(y):=\max\{g_t(y),h(y)\}.
\]
Since no index \(r>t\) is active in \(g(y^\star)\), the subgradient of \(g\)
appearing in the optimality condition belongs to \(\partial g_t(y^\star)\).
Therefore
\(
    0
    \in
    \partial\tilde f_t(y^\star)+\eta(y^\star-x),
\)
so \(y^\star\) is also the unique minimizer of
\(
    y\mapsto
    \tilde f_t(y)+\frac{\eta}{2}\|y-x\|^2.
\)
Moreover, \(\tilde f_t(y^\star)=\tilde f(y^\star)\). Hence
\[
    f(x)
    =
    \inf_y
    \left\{
        \tilde f_t(y)+\frac{\eta}{2}\|y-x\|^2
    \right\},
    \qquad
    \grad f(x)=\eta(x-y^\star).
\]
Both quantities can therefore be computed using only \(v_1,\ldots,v_t\).
This proves the third claim.

It remains to prove the large-query claim. Since
\(
    g(y)
    \le
    \ell\|y\|\) for all $y$,
the norm branch strictly dominates \(g\) whenever
\(
    \|y\|>\frac{1+\frac{\ell}{\sqrt{\budget+1}}}{1-\ell}.
\)
Let
\(
    R_\star
    :=
    \frac{1+\frac{\ell}{\sqrt{\budget+1}}}{1-\ell}
    +
    \frac{1}{\eta}.
\)
For both \(H=\infty\) and \(H=1\), one has \(R_\star<3\). Indeed, if
\(H=\infty\), then \(\ell=1/3\), \(1+\frac{\ell}{\sqrt{\budget+1}}\le4/3\), and
\(1/\eta\le3/10\), so \(R_\star\le2.3<3\). If \(H=1\), then
\(\ell\le1/10\), \(1+\frac{\ell}{\sqrt{\budget+1}}\le1.1\), and \(1/\eta=1\), so
\(
    R_\star
    \le
    \frac{1.1}{0.9}+1
    <3.
\)

Now suppose \(\|x\|>3\). Then \(\|x\|>R_\star\). Set
\(
    y
    =
    \left(1-\frac{1}{\eta\|x\|}\right)x.
\)
Then
\(
    \|y\|
    =
    \|x\|-\frac1\eta
    >
    \frac{1+\frac{\ell}{\sqrt{\budget+1}}}{1-\ell},
\)
so \(h(y)>g(y)\), and therefore
\(
    \partial\tilde f(y)=\partial h(y)=\left\{\frac{y}{\|y\|}\right\}.
\)
Moreover,
\(
    \eta(x-y)=\frac{x}{\|x\|}=\frac{y}{\|y\|}.
\)
Thus \(y\) is the proximal point of \(x\), and the oracle response is
\(
    \grad f(x)
    =
    \eta(x-y)
    =
    \frac{x}{\|x\|},
\)
while
\(
    f(x)
    =
    h(y)+\frac{\eta}{2}\|y-x\|^2
    =
    \|x\|-1-\frac{\ell}{\sqrt{\budget+1}}-\frac{1}{2\eta}.
\)
Both are independent of \(v_1,\ldots,v_{\budget+1}\). Finally, using
\(f^\star\le-\ell/\sqrt{\budget+1}\) and \(\|x\|>3>R_\star\), we have
\[
\begin{aligned}
    f(x)-f^\star
    &\ge
    \|x\|-1-\frac{\ell}{\sqrt{\budget+1}}-\frac{1}{2\eta}
    +
    \frac{\ell}{\sqrt{\budget+1}}  \\
    &>
    \frac{1+\frac{\ell}{\sqrt{\budget+1}}}{1-\ell}
    +
    \frac{1}{\eta}
    -
    1-\frac{\ell}{\sqrt{\budget+1}}
    -
    \frac{1}{2\eta}
    +
    \frac{\ell}{\sqrt{\budget+1}}  \\
    &=
    \frac{\ell\left(1+\frac{\ell}{\sqrt{\budget+1}}\right)}{1-\ell}
    +
    \frac{1}{2\eta}
    +
    \frac{\ell}{\sqrt{\budget+1}}
    \ge
    \frac{\ell}{2\sqrt{\budget+1}}.
\end{aligned}
\]
This proves the fourth claim.
\end{proof}

We then restate the vector-guessing lemma of \citet[Lemma~4]{woodworth2017lower}.

\begin{lemma}
\label{lem:keylowerboundlemma-delta}
Let \(k\in\NN\) with \(k\geq2\), \(\delta\in(0,\frac12]\), and
\(
0<c\leq\frac{1}{2\sqrt{k}}
\), \(V=\{v_1,\ldots,v_k\}\) be a uniformly random orthonormal set in
\(\hilbert{d}\), and \(x_1,\ldots,x_k\) be \(k\)
vectors, each of norm at most \(1\), satisfying the decomposition property
\[
    x_t(V)
    =
    x_t\!\left(
        V_{<t}\indicator{G'_{<t}}
        +
        V\indicator{\neg G'_{<t}}
    \right), \qquad t\le k
\]
where
\[
    G'_{<t}
    =
    \left\llbracket
    \forall r<t,\ \forall j\in\{r,\ldots,k\}:\ 
    |\langle x_r,v_j\rangle|\le \frac{c}{2}
    \right\rrbracket .
\]
If
\begin{equation}\label{choose-dim-in-vector-guessing}
d
\ge
\left\lceil
2k-1+\frac{40k}{c^2}
\log\left(\frac{k^2}{2\delta}\right)
\right\rceil,
\end{equation}
then
\[
\Pr\left(
\forall t\le k,\ \forall j\in\{t,\ldots,k\}:
|\langle x_t,v_j\rangle|\le \frac c2
\right)
\ge 1-2\delta .
\]
\end{lemma}

\begin{proof}
This follows from the argument of
\citet[Equation~(26)]{woodworth2017lower}.
Although that proof uses
\(c=1/(2\sqrt{k})\), the same estimates remain valid for
\(0<c\leq1/(2\sqrt{k})\).
\end{proof}

\begin{lemma}[Randomized first-order lower bounds]
\label{lem:woodworth-nonasym-rand}
For every \(\budget\in\NN\) and every \(\delta\in(0,\frac12]\), the
following two statements hold.

\begin{enumerate}[i.]
    \item There is a probability measure
    \(\probmeasure_{\budget,\delta}^{\rm ns}\) on functions in
    \(\funcls^{\rm ns}_{d_{\budget,\delta}^{\rm ns}}\) such that,
    for every randomized first-order algorithm
    \(\alg\in\algclsrand^{(1)}\),
    \[
        \Pr_{\alg,\,f\sim\probmeasure_{\budget,\delta}^{\rm ns}}
        \left(
        f(\alg[f]_{\budget})-f^\star
        \ge
        \frac{1}{6\sqrt{2}\budget^{1/2}}
        \right)
        \ge
        1-2\delta ,
    \]
    \item There is a probability measure
    \(\probmeasure_{\budget,\delta}^{\rm sm}\) on functions in
    \(\funcls^{\rm sm}_{d_{\budget,\delta}^{\rm sm}}\) such that, for
    every randomized first-order algorithm
    \(\alg\in\algclsrand^{(1)}\),
    \[
        \Pr_{\alg,\,f\sim\probmeasure_{\budget,\delta}^{\rm sm}}
        \left(
        f(\alg[f]_{\budget})-f^\star
        \ge
        \frac{1}{80\budget^{2}}
        \right)
        \ge
        1-2\delta.
    \]
\end{enumerate}
\end{lemma}

\begin{proof}
Fix \(H\in\{1,\infty\}\), let
\[
    \ell_H
    =
    \begin{cases}
        \frac13, & H=\infty,\\[2mm]
        \frac{1}{10(\budget+1)^{3/2}}, & H=1,
    \end{cases}
    \qquad
    \eta_H
    =
    10(\budget+1)^{3/2}\ell_H .
\]
Then
\(
    \frac{\ell_H}{\eta_H}
    =
    \frac{1}{10(\budget+1)^{3/2}}.
\)
We will apply \Cref{lem:keylowerboundlemma-delta} to rescaled surrogate
points, so take
\(
    k=\budget+1,
    c
    =
    \frac{2\ell_H}{3\eta_H}
    =
    \frac{1}{15(\budget+1)^{3/2}}.
\)
Notice that
\(
c=\frac{1}{15k^{3/2}}
\leq
\frac{1}{2\sqrt{k}},
\)
so \Cref{lem:keylowerboundlemma-delta} applies. Choose an integer \(d=d_{\budget,\delta}\) satisfying \eqref{choose-dim-in-vector-guessing}.

Let \(v_1,\ldots,v_{\budget+1}\) be a uniformly random orthonormal set in \(\hilbert{d}\),
and let \(f=f_V\) be the hard instance defined in
\eqref{eq:hard-func-woodworth} with parameters \(\ell_H,\eta_H\). Denote
the induced law of \(f_V\) by \(\probmeasure_{\budget,\delta}^{H}\). By
\Cref{lem:hard-instance-unbounded-queries},
\(
    \probmeasure_{\budget,\delta}^{\infty}\)
is supported on 
    \(\funcls_{d}^{\rm ns},\)
\(\probmeasure_{\budget,\delta}^{1}\)
is supported on 
\(\funcls_{d}^{\rm sm}.
\)

We first consider a deterministic first-order algorithm. Let \(x_t\) be
its \(t\)-th query point for \(t=1,\ldots,\budget\), and set
\(
    x_{\budget+1}:=\alg[f]_{\budget}
\)
to be its final output. Define the bounded surrogate points
\[
    z_t
    :=
    \begin{cases}
        x_t/3, & \|x_t\|\le3,\\
        0, & \|x_t\|>3,
    \end{cases}
    \qquad
    t=1,\ldots,\budget+1.
\]
Then \(\|z_t\|\le1\) for every \(t\).

We claim that \(z_1,\ldots,z_{\budget+1}\) satisfy the decomposition condition in
\Cref{lem:keylowerboundlemma-delta}. Define
\[
    G'_{<t}
    :=
    \left\llbracket
    \forall s<t,\ \forall j\in\{s,\ldots,\budget+1\}:\ 
    |\langle z_s,v_j\rangle|
    \le
    \frac{c}{2}
    \right\rrbracket .
\]
On \(G'_{<t}\), consider any previous query \(x_s\), \(s<t\). If
\(\|x_s\|>3\), then the fourth property of
\Cref{lem:hard-instance-unbounded-queries} implies that the oracle
response at \(x_s\) is independent of \(v_1,\ldots,v_{\budget+1}\). If
\(\|x_s\|\le3\), then \(z_s=x_s/3\), and the event \(G'_{<t}\) gives
\(
    |\langle x_s,v_j\rangle|
    =
    3|\langle z_s,v_j\rangle|
    \le
    \frac{3c}{2}
    =
    \frac{\ell_H}{\eta_H}\)
for all \(j\in\{s,\ldots,\budget+1\}.\)

By the third property of \Cref{lem:hard-instance-unbounded-queries}, the
oracle response at such an \(x_s\) can be computed using only
\(v_1,\ldots,v_s\). Hence, by induction on \(t\), the transcript before
producing \(x_t\) depends only on \(v_1,\ldots,v_{t-1}\), and so \(z_t\)
also depends only on \(v_1,\ldots,v_{t-1}\) on \(G'_{<t}\). This is
precisely the decomposition condition.

Applying \Cref{lem:keylowerboundlemma-delta} to
\(z_1,\ldots,z_{\budget+1}\), with \(k=\budget+1\) and
\(c=2\ell_H/(3\eta_H)\), gives that, with probability at least
\(1-2\delta\) over the random draw of \(V\),
\(
    \forall t\le \budget+1,\ \forall j\in\{t,\ldots,\budget+1\}\), we have
   \( |\langle z_t,v_j\rangle|
    \le
    \frac{c}{2}.
\)
On this event, we prove the desired lower bound for the final output
\(x_{\budget+1}=\alg[f]_{\budget}\).

If \(\|x_{\budget+1}\|>3\), then the fourth property of
\Cref{lem:hard-instance-unbounded-queries} gives directly
\(
    f(x_{\budget+1})-f^\star
    \ge
    \frac{\ell_H}{2\sqrt{\budget+1}}.
\)
If \(\|x_{\budget+1}\|\le3\), then \(z_{\budget+1}=x_{\budget+1}/3\), and the vector-guessing event gives
\(
    |v_{\budget+1}^T x_{\budget+1}|
    =
    3|v_{\budget+1}^T z_{\budget+1}|
    \le
    \frac{3c}{2}
    =
    \frac{\ell_H}{\eta_H}.
\)
The second property of \Cref{lem:hard-instance-unbounded-queries} therefore
again yields
\(
    f(x_{\budget+1})-f^\star
    \ge
    \frac{\ell_H}{2\sqrt{\budget+1}}.
\)
Thus, for deterministic algorithms, with probability at least
\(1-2\delta\),
\(
    f(\alg[f]_{\budget})-f^\star
    \ge
    \frac{\ell_H}{2\sqrt{\budget+1}}.
\)

When \(H=\infty\), this gives
\(
    f(\alg[f]_{\budget})-f^\star
    \ge
    \frac{1}{6\sqrt{\budget+1}}
    \ge
    \frac{1}{6\sqrt2}\,\budget^{-1/2}.
\)
When \(H=1\), this gives
\(
    f(\alg[f]_{\budget})-f^\star
    \ge
    \frac{1}{20(\budget+1)^2}
    \ge
    \frac{1}{80}\,\budget^{-2}.
\)

Now let \(\alg\) be randomized. Conditioning on its random seed \(\xi\)
turns \(\alg\) into a deterministic first-order algorithm. The preceding
argument applies for every fixed \(\xi\). Integrating over \(\xi\) gives
the same lower bound with respect to the joint randomness of the algorithm
and the random draw of \(f\).
\end{proof}
Applying \Cref{thm.lower-bound-fixed-function-convex-rand} with
\Cref{lem:convex-constrained-error-metric,lem:woodworth-nonasym-rand}
proves both statements of \Cref{coro:convex-fixed-rand}.

\section{Supplement to \texorpdfstring{\Cref{sec.lower bound}}{Section \ref*{sec.lower bound}}}
\subsection{Useful lemma on the hardness of estimation}\label{sec:hardness-of-estimation}

The developments in this section use the standard notation
$$
I(X ; Y)=H(X)-H(X \mid Y)=H(Y)-H(Y \mid X)
$$
for mutual information $I(X ; Y)$, entropy $H(X)=\mathbb{E}_{x \sim X} \log {\frac{1}{\pr(X=x)}}$ and conditional entropy $H(X \mid Y)=\mathbb{E}_{x, y \sim X, Y} \log \frac{1}{\pr(X=x \mid Y=y)}$. We also write
$$
h_2(q):=H(\mathsf {Bernoulli}(q))=q \log \frac{1}{q}+(1-q) \log \frac{1}{1-q}
$$
for the binary entropy function. Our bounds hinge on the following classical result.
\begin{fact}\label{fact.fano}
(Fano's inequality~\cite{fano1961transmission}). Let $V$ be a random variable taking $n$ values, and let $\hat{V}$ be some estimator of $V$. Then
$$
h_2(\pr(V \neq \hat{V}))+\pr(V \neq \hat{V}) \log (n-1) \geq H(V \mid \hat{V}) .
$$
\end{fact}
First, we introduce the following lemma, which clarifies the hardness of sharpening a biased Bernoulli prior~(adapted from \citet[Lemma~2]{carmon2024price}). In the following, all three-valued discrete random variables have three possible values $\{-1,0,1\}$.

For brevity, we will denote $S\ind{1}, S\ind{2}, \dots, S\ind{\budget}$ by $S\ind{1:\budget}$ and $Q\ind{1}, Q\ind{2}, \dots, Q\ind{\budget}$ by $Q\ind{1:\budget}$.

% log is natural log here
\begin{lemma}\label{lem:V-information-theory}
For any $\budget \in \mathbb{N}$, $\epsilon \in (0,1)$, $w\in (0,1)$ let $V$ be a Rademacher random variable, $S\ind{1:\budget} \simIID \threval(w,\epsilon, V)$ and $Q\ind{1:\budget}$ be a sequence of random variables. 
Suppose that
 \begin{equation}\label{eq.upper bound of epsilon}
\epsilon \le \sqrt{\frac{h_2(1/2) - h_2(1/3)}{w\budget}}.
\end{equation}
Then, for every randomized estimator $\hat{V}$ such that $V \independent \hat{V} \mid S\ind{1:\budget}, Q\ind{1:\budget}$, we have
\[
\pr(\hat{V} \neq V) \geq \frac{1}{3} - \sum_{t=1}^{\budget} \pr(Q\ind{t} \neq 0).
\]
\end{lemma}

\begin{proof}
Since $\hat{V}$ is a randomized estimator, it is a function of the form
$\hat{V}(S\ind{1:\budget},Q\ind{1:\budget},U)$
where $U$ is a uniform random variable.
We will analyze the randomized estimator:
\[
\hat{W}(S\ind{1:\budget}, U) := \hat{V}(S\ind{1:\budget}, \mathbf{0}, U)
\]
which only depends on $S\ind{1:\budget}$ instead of $Q\ind{1:\budget}$. 
This makes $\hat{W}$ easier to analyze than $\hat{V}$.
Moreover, by a union bound,
\begin{flalign*}
\pr(\hat{V} \neq V) &\ge \pr(\hat{W} \neq V, \hat{W} = \hat{V})
\ge \pr(\hat{W} \neq V) - \pr(\hat{W} \neq \hat{V}) \\
&\ge \pr(\hat{W} \neq V) - \pr(Q\ind{1:\budget} \neq \mathbf{0}) \ge \pr(\hat{W} \neq V) - \sum_{t=1}^\budget \pr(Q\ind{t} \neq 0).
\end{flalign*}
The remainder of the proof shows $\pr(\hat{W} \neq V) \ge 1/3$. In particular, we will show $I(V ; \hat{W}) \leq w\epsilon^2 \budget$ and then use Fano's inequality to provide a bound on the probability.
By the data processing inequality~\cite[Theorem 2.8.1]{CoverThomas1991} and the chain rule for mutual information~\cite[Theorem 2.5.2]{CoverThomas1991},    
\begin{equation}
    \label{eq.upper bound mutual info}
    I(\hat{W} ; V) \leq I\left(S\ind{1:\budget} ; V\right)=\sum_{t=1}^\budget\left[H\left(S\ind{t} \mid S\ind{1:t-1}\right)-H\left(S\ind{t} \mid S\ind{1:t-1}, V\right)\right].
\end{equation}
For $H\left(S\ind{t} \mid S\ind{1:t-1}\right)$, using the facts that conditioning decreases entropy, and that $\pr(S\ind{t} = \pm 1) = \frac{w}{2}$ and $\pr(S\ind{t} =0) = 1-w$ we get
\begin{flalign}\label{bound:entropy-conditional-on-S}
H(S\ind{t} \mid S\ind{1:t-1} ) \le H(S\ind{t}) = (1 - w) \log \frac{1}{1-w} + w \log \frac{2}{w}.
\end{flalign}
Similarly, we have
\begin{equation}\label{bound:entropy-conditional-on-S-and-V}
    \begin{aligned}
    &H\left(S\ind{t} \mid S\ind{1:t-1}, V\right) = H(S\ind{t} \mid V)  \\
    &= \E\left[ (1 - w) \log \frac{1}{1-w} + \frac{w (1-\epsilon V)}{2} \log \frac{2}{w (1-\epsilon V)}  + \frac{w (1+\epsilon V)}{2} \log \frac{2}{w (1+\epsilon V)} \right]  \\
    &= (1 - w) \log \frac{1}{1-w} +  \frac{w (1-\epsilon)}{2}  \log \frac{2}{w (1-\epsilon)} + \frac{w (1+\epsilon)}{2} \log \frac{2}{w (1+\epsilon)}.  
    \end{aligned}
\end{equation}
Thus, substituting these two inequalities \eqref{bound:entropy-conditional-on-S} and \eqref{bound:entropy-conditional-on-S-and-V} into \eqref{eq.upper bound mutual info} and using $\log(1+x) \le x$ gives
\begin{flalign*}
I(\hat{W} ; V) &\le \budget w \left( \log \frac{2}{w} - \frac{(1-\epsilon)}{2}  \log \frac{2}{w (1-\epsilon)} - \frac{(1+\epsilon)}{2} \log \frac{2}{w (1+\epsilon)} \right) \\
&= \budget w \left( \frac{(1-\epsilon)}{2} \log(1-\epsilon) + \frac{(1+\epsilon)}{2} \log(1+\epsilon) \right) \le \budget w \left(  \frac{(\epsilon - 1) \epsilon}{2} + \frac{(1+\epsilon)\epsilon}{2} \right) \\
&= \budget w \epsilon^2.
\end{flalign*}
It remains to use our bound on the mutual information to bound
 $q := \pr(V \neq \hat{W})$.
Applying Fano's inequality (\Cref{fact.fano}) with $n=2$ gives 
\[
h_2(q) \geq H(V \mid \hat{W}) = H(V) - I(\hat{W} ; V) = h_2(1/2) - \budget w \epsilon^2 \geq \log(2) - \budget w \epsilon^2.
\]
Applying \eqref{eq.upper bound of epsilon} gives 
\[
h_2(q) \ge \log(2) - \budget w \epsilon^2 \ge h_2(1/3)
\]
Since $h_2(q)$ is increasing on $[0,1/2)$ it follows that $q \ge 1/3$.
\end{proof}

\bibliographystyle{abbrvnat}
\bibliography{references}
\end{document}